\documentclass[11pt]{article}

\usepackage{amsmath,amsthm}

\usepackage{amssymb,latexsym}

\usepackage{enumerate}

\usepackage{amsmath}
\usepackage{amssymb}
\def\toP{{\scriptstyle\,\buildrel{\P}\over{\hbox to
0.6cm{\rightarrowfill}}\,}}
\newcommand{\R}{{\mathbb R}}

\newcommand{\BR}{{\mathbb R}}

\newcommand{\sgn}{\mbox{\rm sgn}}

\newcommand{\N}{{\mathbb N}}
\newcommand{\LL}{{\mathbb L}}

\newcommand{\HH}{{\cal H}}

\newcommand{\MM}{{\cal M}}
\newcommand{\VV}{{\cal V}}
\newcommand{\EE}{{\cal E}}

\newcommand{\Rd}{{{{\mathbb R}^d}}}

\newcommand{\FF}{{\cal F}}

\newcommand{\esssup}{\mathop{\mathrm{ess\,sup}}}

\newtheorem{theorem}{\bf Theorem}[subsection]
\newtheorem{proposition}[theorem]{\bf Proposition}%[subsection]
\newtheorem{lemma}[theorem]{\bf Lemma}%[subsection]

\theoremstyle{definition}
\newtheorem{definition}[theorem]{Definition}
\newtheorem{remark}[theorem]{\bf Remark}%[subsection]

\newcommand{\nsubsection}{\setcounter{equation}{0}\subsection}

\begin{document}
\title{Reflected backward stochastic differential equations with \\
two optional barriers}
\author{Tomasz Klimsiak, Maurycy Rzymowski and Leszek S\l omi\'nski}
\date{}
\maketitle

\begin{abstract}
We consider reflected backward stochastic differential equations with
two general optional barriers. The solutions to these equations
have the so-called regulated trajectories, i.e
trajectories with left and right finite limits. We prove the existence
and uniqueness of $\LL^p$ solutions, $p\geq 1$, and show that the solutions   may be
approximated by a modified penalization method.
\end{abstract}
\footnotetext{T. Klimsiak, M. Rzymowski and L. S\l omi\'nski: Faculty of
Mathematics and Computer Science, Nicolaus Copernicus University,
Chopina 12/18, 87-100 Toru\'n, Poland. E-mails:
tomas@mat.umk.pl, maurycyrzymowski@mat.umk.pl, leszeks@mat.umk.pl}

\nsubsection{Introduction}

In the present paper we study the existence, uniqueness and approximations of $\LL^p$, $p\ge1$, solutions of  reflected backward stochastic differential equations (RBSDEs) with monotone generator $f:\Omega\times[0,T]\times\BR\times \BR^d\to \BR$  and  two optional
barriers $L,U$  satisfying  the so-called generalized Mokobodzki condition.

The notion of RBSDE with one reflecting continuous barrier was introduced by El Karoui, Kupoudjian, Pardoux, Peng and Quenez \cite{EKPPQ}, who proved the existence and uniqueness of solutions of equations with  Lipschitz continuous  generator and  square-integrable data.
RBSDEs with two continuous barriers  were for the first time considered  by Cvitani\'c and Karatzas \cite{CvitanicKaratzas} under the same assumptions on the generator and the data.
In \cite{CvitanicKaratzas}, a solution is  a triple  $(Y,Z,R)$ of $\mathbb{F}$-progressively measurable processes such that $Y$ is  continuous and
\begin{equation}\label{eq01}
Y_t=\xi+\int_t^Tf(r,Y_r,Z_r)\,dr+\int_t^T\,dR_r-\int_t^T Z_r\,dB_r,\quad t\in [0,T],
\end{equation}
where  $B$ is a standard $d$-dimensional Brownian motion and
$\mathbb{F}$ is the standard augmentation of the natural filtration
generated by $B$. Moreover, it is required that
\begin{equation}\label{eq015}
L_t\le Y_t\le U_t,\quad t\in [0,T],
\end{equation}
and $R$  is a continuous process of finite variation such that $R_0=0$ and  the following minimality condition is satisfied:
\begin{equation}\label{eq02}\int_0^T(Y_r-L_r)\,dR^+_r+\int_0^T(U_r-Y_r)\,dR^-_r=0.
\end{equation}
Here $R^+, R^-$ stand for the positive and negative part of the Jordan decomposition of the measure $dR$.  In \cite{CvitanicKaratzas} the existence and uniqueness of a solution is proved. Note also that in \cite{CvitanicKaratzas,EKPPQ} important  connections between solutions of RBSDEs and suitably defined optimal stopping problems were established.

Since the pioneering works \cite{CvitanicKaratzas,EKPPQ} reflected BSDEs have been intensively studied by many authors. The results of \cite{CvitanicKaratzas,EKPPQ} were
generalized to equations with $\LL^p$ data and c\`adl\`ag barriers (see, e.g., \cite{bdh,Hamadene2,HH,HP,kl1,LMX,LX1,RS2}). The assumption that the barriers are c\`adl\`ag  implies that the first component $Y$ of a solutions is also  a c\`adl\`ag process.
Therefore this assumption is sometimes too strong when one think on applications of RBSDEs  to optimal stopping problems, because it is known that in general solutions of such problems have merely regulated trajectories (see \cite{EK}). It is worth noting here  that in \cite{kl,PX}
RBSDEs with non-c\`adl\`ag (progressively measurable) barriers and c\`adl\`ag solutions are considered. However, in the definition of a solutions adopted in  \cite{kl,PX} its first  component $Y$  need not satisfy (\ref{eq015}), but satisfies an essentially weaker condition
saying that $L_t\leq Y_t\leq U_t$ for a.e. $t\in [0,T]$.

To our knowledge, RBSDEs with barriers which are not c\`adl\`ag  and whose solution satisfies (\ref{eq015}) are treated only in the papers \cite{BO,GIOOQ,GIOQ,KRzS}. Among them, only  \cite{GIOQ} deals with equations with two barriers.
 In the present paper we generalize the  existence and
uniqueness results from \cite{GIOQ}  in several directions.  We consider the case of
$L^p$-data with $p\ge 1$ (in \cite{GIOQ} only the case of $p=2$
is considered). As for the generator, we assume that it is
Lipschitz continuous with respect to $z$ and only continuous and
monotone with respect to $y$ (in \cite{GIOQ} it is assumed that
$f$ is Lipschitz continuous with respect to $y$ and $z$). Moreover, we assume  that the generator and the barriers  satisfy the so-called generalized Mokobodzki condition which says that there exists a semimartingale
$X\in\MM_{loc}+\VV^p$ such that $L_t\le X_t\le U_t$,  $t\in[0,T]$, and
\begin{equation}
\label{eq.intr.m}
E\Big(\int_0^T|f(r,X_r,0)|\,dr\Big)^p+|X|_p<\infty,
\end{equation}
where $|X|_p:=(E\sup_{t\le T}|X_t|^p)^{1/p}$ for $p>1$ and $|X|_1:=\sup_{\tau\in\Gamma}E|X_\tau|$ (here $\MM_{loc}$
is the space of local martingales and $\VV^p$ is the space of finite variation processes with $p$-integrable variation, and $\Gamma$  denotes  the set of all
$\mathbb{F}$-stopping times). In \cite{GIOQ} the standard Mokobodzki condition is assumed. It says  that $L\le X\le U$ for some semimartingale $X\in\MM_{loc}+\VV^2$ such that $|X|_2<\infty$. This condition  automatically  implies (\ref{eq.intr.m}) with $p=2$ in  case $f$ is Lipschitz  continuous.

The assumptions on $\xi$ and $f$ adopted in the present paper are the same as in our previous paper \cite{KRzS} devoted to equations with one lower barrier, and our definition of a solution is a counterpart to the definition introduced in \cite{KRzS}. For a process $\eta$, let $\Delta^+\eta_t=\eta_{t+}-\eta_t$, $\Delta^-\eta_t=\eta_t-\eta_{t-}$, i.e. $\Delta^+\eta_t$, $\Delta^-\eta_t$ denote the right and left jump of $\eta$ at $t$. Our definition says that a triple $(Y,Z,R)$ of $\mathbb F$-progressively measurable processes is a  solution of RBSDE on the interval $[0,T]$ with terminal time $\xi$, right-hand side $f$
and optional barriers $L, U$ (RBSDE$(\xi,f,L,U)$ for short) if $Y,R$ are regulated processes,  $R$ is a finite variation process with $R_0=0$, (\ref{eq01}) and  (\ref{eq015})  hold true, and the following minimality condition  is satisfied:
\begin{align}
\label{eq.intr.min}
\nonumber\int_0^T (Y_{r-}&-\limsup_{s\uparrow r}L_{s})\,dR^{+,*}_r +\sum_{r<T}(Y_r-L_r)\Delta^+R^+_r\\
&+\int_0^T (\liminf_{s\uparrow r}U_{s}-Y_{r-})\,dR^{-,*}_r+\sum_{r<T}(U_r-Y_r)\Delta^+R^-_r=0,
\end{align}
where $R^{+,*}, R^{-,*}$ are  c\`adl\`ag parts of the processes $R^+, R^-$.  We show that if the barriers $L,U$ are regulated, then
$\Delta^-R^+_t=(Y_t-L_{t-})^-$, $\Delta^-R^-_t=(Y_t-U_{t-})^+$,
and
$\Delta^+R^+_t=(Y_{t+}-L_{t})^-$, $\Delta^+R^-_t=(Y_{t+}-U_{t})^+$.
If the barriers are c\`adl\`ag (resp. c\`agl\`ad)
then $\Delta^+R\equiv 0$ (resp. $\Delta^-R=0$). Consequently, if $L,U$ are continuous, then  condition (\ref{eq.intr.min})
reduces to (\ref{eq02}). Moreover, if the barriers are c\`adl\`ag, then condition (\ref{eq.intr.min}) reduces to the minimality
condition considered in \cite{Hamadene2}.
In  the present paper,  we generalize the existence, uniqueness and approximation results proved in \cite{KRzS}. It is worth pointing out, however, that the proofs  are essentially more complicated and in many points different from those in \cite{KRzS}.

Our main results are proved in Sections 3 and 4.
In Section \ref{sec3}, we consider equations with general two optional barriers (they need not be regulated). We show  that there exists a unique solution $(Y,Z,R)$ to RBSDE$(\xi,f,L,U)$ such that $|Y|_p+E|R^+_T|^p+E|R^-_T|^p<\infty$ and
$E(\int_0^T|Z_r|^2\,dr)^{p/2}<\infty$ in case $p>1$, and $E(\int_0^T|Z_r|^2\,dr)^{q/2}<\infty$ for $q\in (0,1)$ in case  $p=1$. In case $p=1$,   we  assume additionally that $f$ satisfies condition  (Z) introduced in the paper \cite{bdh} devoted to usual (nonreflected) BSDEs. The proof of the existence part is divided into two steps. In the first step, we assume that $f$ does not depend on $z$ and we solve the following decoupling system
\[
\left\{
\begin{array}{ll} Y^1_t=\esssup_{t\le\tau\le T}E(Y^2_\tau+\int_t^\tau f(r,Y^1_r-Y^2_r)\,dr+L_\tau\mathbf{1}_{\tau<T}+\xi\mathbf{1}_{\tau=T}|\FF_t),
\smallskip\\
 Y^2_t=\esssup_{t\le\tau\le T}E(Y^1_\tau\mathbf{1}_{\tau<T}-U_\tau\mathbf{1}_{\tau<T}|\FF_t).
\end{array}
\right.
\]
This system  may be equivalently formulated as a system of RBSDEs with lower optional barriers (see \cite{BO,GIOOQ,KRzS}).
Putting  $Y=Y^1-Y^2$, we obtain  a solution of RBSDE$(\xi,f,L,U)$. Note that in the linear case, i.e. when $f$ does not depend on $y$ as well, this method was considered in the context of Dynkin games problem by Bismut \cite{Bismut1,Bismut2}
(see also \cite{KQC,N}). Next, to solve the nonlinear problem,  we apply a fixed point argument in case $p>1$, and Picard iteration procedure in case $p=1$.

In Section \ref{sec4}, under the additional assumption that the barriers $L,U$ are regulated, we propose another
approach to the existence problem. We consider two penalization schemes based on BSDEs
with penalty term and RBSDEs with one barrier and penalty term. In the first one, we show that there exists a unique solution $(Y^n,Z^n)$ of generalized BSDE of the form
\begin{align}\label{r.4.intr}
\nonumber
Y^n_t&=\xi+\int^T_t f(r,Y^n_r,Z^n_r)\,dr-\int^T_t Z^n_r\,dB_r+n\int^T_t(Y^n_r-L_r)^-\,dr\\
&\quad+\sum_{t\le\sigma_{n,i}<T}(Y^n_{\sigma_{n,i}+}-L_{\sigma_{n,i}})^-
-n\int^T_t(Y^n_r-U_r)^+\,dr-\sum_{t\le\tau_{n,i}<T}(Y^n_{\tau_{n,i}+}-U_{\tau_{n,i}})^+,
\end{align}
where $\{\{\sigma_{n,i}\}\}$ (resp. $\{\{\tau_{n,i}\}\}$) is a suitably defined array of stopping times exhausting  the right jumps of $L$ (resp. $U$). We prove that
\begin{equation}
\label{intr.1a}
Y^n_t\rightarrow Y_t,\,\quad  t\in [0,T].
\end{equation}
Moreover, for every $\gamma\in (0,2)$,
\begin{equation}
\label{intr.1b}
E(\int_0^T|Z^n_r-Z_r|^\gamma\,dr )^{p/2}\rightarrow 0
\end{equation}
if $p>1$, and
\begin{equation}
\label{intr.1c}
E(\int_0^T|Z^n_r-Z_r|^\gamma\,dr )^{q/2}\rightarrow 0,\quad  q\in (0,1),
\end{equation}
if $p=1$. We also prove that if $\Delta^-R=0$, then
$|Y^n-Y|_p\rightarrow 0,$
and  (\ref{intr.1b}) holds true with $\gamma=2$. To prove (\ref{intr.1a})--(\ref{intr.1c})
we first show the convergence of penalization schemes based on RBSDEs. In this scheme, $(\bar Y^n,\bar Z^n,\bar K^n)$ (resp. $(\underline Y^n,\underline Z^n,\underline A^n)$) is a solution to reflected BSDE with upper barrier $U$ (resp. lower barrier $L$) and the generator being a sum of $f$ and an additional penalty term (depending on $n$) involving $L$ (resp. $U$) and the right-side jumps of $L$ (resp. $U$). We prove that
$(\bar Y^n,\bar Z^n,\bar K^n)$,  $(\underline Y^n,\underline Z^n,\underline A^n)$ converge to $(Y,Z,R)$ in the sense of (\ref{intr.1a})--(\ref{intr.1c}) and
\[
\bar Y^n_t\le Y^n_t\le \underline Y^n_t,\quad t\in [0,T].
\]
The advantage  of these approximations is that $\{\bar Y^n\}$ is nondecreasing and $\{\underline Y^n\}$ is nonincreasing.

\nsubsection{Preliminaries}

Let $B$ be a standard Wiener process defined on some probability space $(\Omega,\mathcal{F},P)$ and let $\mathbb{F}=\{\mathcal{F}_t,\,t\in[0,T]\}$ be the standard augmentation of the filtration generated by $B$. Recall that a function $y:[0,T]\to\Rd$ is called regulated if for every $t\in[0,T)$  the limit $y_{t+}=\lim_{u\downarrow t}y_u$ exists, and for every $s\in(0,T]$ the limit $y_{s-}=\lim_{u\uparrow s}y_u$ exists. For any regulated function $y$ on $[0,T]$, we set $\Delta^{+}y_t=y_{t+}-y_t$ if $0\leq t<T$, and  $\Delta^{-}y_s=y_s-y_{s-}$ if  $0<s\leq T$. We also set $\Delta^{+}y_T$ $=\Delta^{-}y_0$ $=0$, $\Delta y_t=\Delta^+y_t+\Delta^-y_t$, $t\in[0,T]$  and $y^{\oplus}_t=y_{t+}$ if $t<T$, and $y^{\oplus}_T=y_T$. Note that $y^{\oplus}$ is a c\`adl\`ag function such that $y^{\oplus}_t=\Delta^+y_t+y_t$, $t\in[0,T]$. It is known that each regulated function is bounded and has at most countably many discontinuities (see, e.g., \cite[Chapter 2, Corollary 2.2]{dn}).

For $x\in\R^d$, $z\in\R^{d\times n}$, we set
$|x|^2=\sum^d_{i=1}|x_i|^2$, $\|z\|^2=\mbox{trace}(z^*z)$.
$\langle\cdot,\cdot\rangle$ denotes the usual scalar product in
$\R^d$  and $\sgn(x)={\bf 1}_{\{x\neq0\}}{x}/{|x|}$.

For a fixed stopping time $\tau$, we denote by $\Gamma_\tau$  the set of all
$\mathbb{F}$-stopping times taking values in $[\tau,T]$. We put $\Gamma:=\Gamma_0$.
We denote by $\mathbb{L}^p$, $p>0$,  the space of random variables $X$ such
that $\|X\|_p\equiv E(|X|^p)^{1\wedge1/p}<\infty$. We denote by $\mathcal{S}$
the set of all $\mathbb{F}$-adapted regulated processes, and by $\mathcal{S}^p$,
$p>0$, the subset of  $Y\in\mathcal{S}$ such that $E\sup_{0\le t\le T}|Y_t|^p<\infty$.
Given a  regulated $\mathbb F$-adapted process $X$,  we set,
\[
|X|_p=\left\{
\begin{array}{ll} (E\sup_{t\le T}|X_t|^p)^{1\wedge (1/p)},&
\mbox{for}\,\, p\neq1,
\smallskip\\
\sup_{\tau\in\Gamma}E|X_\tau|,& \mbox{for}\,\, p=1.
\end{array}
\right.
\]
$\mathcal{H}$ is the set of $\mathbb{F}$-progressively measurable processes
$X$ such that $
P\big(\int^T_0|X_t|^2\,dt<\infty\big)=1,
$
and $\mathcal{H}^p$, $p>0$, is the set of all $X\in\mathcal{H}$
such that $\|X\|_{\mathcal
{H}^p}\equiv\|(\int_0^T|X_s|^2\,ds)^{1/2}\|_p<+\infty$.

We say that an $\mathbb{F}$-progressively measurable process $X$
is of class (D) if the family $\{X_{\tau},\,\tau\in\Gamma\}$ is
uniformly integrable. We equip the
space of processes of class (D) with the norm
$|\cdot|_1$.

For $\tau\in\Gamma$, we denote by $[[\tau]]$ the set
$\{(\omega,t):\,\tau(\omega)=t\}$. An increasing sequence
$\{\tau_k\}\subset\Gamma$ is called a chain if
$
\forall{\omega\in\Omega}$ $
\exists{n\in\mathbb{N}}$  $\forall{k\ge n}$ $\tau_k(\omega)=T$.

$\mathcal{M}$ (resp. $\mathcal{M}_{loc}$) is the set of all
$\mathbb{F}$-martingales (resp. local martingales). $\MM^p$, $p\ge1$, denotes the space of all $M\in \MM$
such that
$
E([M]_T)^{p/2}<\infty,
$
where $[M]$ stands for  the quadratic variation of $M$.

$\mathcal{V}$ (resp. $\mathcal{V}^+$) denotes the space of
$\mathbb{F}$-progressively measurable process of finite variation
(resp. increasing) such that $V_0=0$,  and $\mathcal{V}^p$ (resp.
$\mathcal{V}^{+,p}$), $p\ge1$,  is the set of processes
$V\in\mathcal{V}$ (resp. $V\in\mathcal{V}^+$) such that
$E|V|^p_T<\infty$, where $|V|_T$ denotes the total variation of
$V$ on $[0,T]$.   For $V\in \mathcal{V}$, we denote by  $V^*$ the
c\`adl\`ag part of the process $V$, and by  $V^d$ its purely
jumping part consisting of right jumps, i.e.
\[
V^d_t=\sum_{s<t}\Delta^+V_s,\quad V^*_t=V_t-V^d_t,\quad t\in
[0,T].
\]

Let $V^1,V^2\in \mathcal{V}$. We write $dV^1\le dV^2$ if
$dV^{1,*}\le dV^{2,*}$ and $\Delta^+V^1\le\Delta^+V^2$ on $[0,T]$.

In the whole paper all relations between random variables hold
$P$-a.s. For process $X$ and $Y$, we  write $X\le Y$ if $X_t\le Y_t$,
$t\in[0,T]$.

We assume that $V\in \VV$, the barriers $L,U$  are $\mathbb{F}$-adapted optional processes, $L_T\le\xi\le U_T$, and the generator is a map \[
\Omega\times[0,T]\times\BR\times \BR^d\ni (\omega,t,y,z)\mapsto f(\omega,t,y,z)\in \BR,
\]
which is $\mathbb F$-adapted  for fixed $y,z$. %Let $f:\Omega\times[0,T]\times\mathbb{R}\times\mathbb{R}^d\rightarrow\mathbb{R}$.
 We will need the following assumptions.
\begin{enumerate}
\item[(H1)]There is $\lambda\ge0$ such that
$|f(t,y,z)-f(t,y,z')|\le\lambda|z-z'|$ for $t\in[0,T]$, $y\in\R$,
$z,z'\in\Rd$.
\item[(H2)]There is $\mu\in\R$ such that
$(y-y')(f(t,y,z)-f(t,y',z))\leq\mu(y-y')^2$ for $t\in[0,T]$,
$y,y'\in\R$, $z\in\Rd$.
\item[(H3)] $\xi,\, \int_0^T|f(r,0,0)|\,dr,\,|V|_T\in\mathbb{L}^p$.
\item[(H4)] For every $(t,z)\in[0,T]\times\mathbb{R}^d$ the mapping $\mathbb{R}\ni y\rightarrow f(t,y,z)$ is continuous.
\item[(H5)] $[0,T]\ni t\mapsto f(t,y,0)\in L^1(0,T)$ for every $y\in\mathbb{R}$,
\item[(H6)] There exists a process $X\in \mathcal{M}_{loc}+\mathcal{V}^p$ such that $X\in\mathcal{S}^p$, $L\le X\le U$ and $\int_0^T |f(r,X_r,0)|\,dr\in \mathbb{L}^p$.
\item[(H6*)] There exists a process $X\in \mathcal{M}_{loc}+\mathcal{V}^1$ such that $X$ is of class (D),
%$X\in\mathcal{S}^q$, $q\in(0,1)$,
$L\le X\le U$ and $\int_0^T |f(r,X_r,0)|\,dr\in \mathbb{L}^1$,
\item[(Z)] There exists a  progressively measurable process $g$ and $\gamma\ge 0,\, \alpha\in [0,1)$
such that
\[
|f(t,y,z)-f(t,y,0)|\le\gamma(g_t+|y|+|z|)^\alpha,\quad t\in [0,T], \,y\in\BR,\, z\in\BR^d.
\]
\end{enumerate}

\begin{remark}
If $X\in \mathcal S$ and $X$ is of class (D), then $X\in\mathcal{S}^q$ for $q\in(0,1)$.
To see this, we let $\sigma_a=\inf\{t\ge 0: |X_t|>a\}\wedge T$. Then for $q\in (0,1)$ and $b>0$,
\begin{align*}
E\sup_{t\le T}|X_t|^q=E\sup_{t\le T}|X^\oplus_t|^q&\le b +\int_b^\infty P(\sup_{t\le T}|X^\oplus_t|^q>a)\,da\\&
\le
b+\int_b^\infty \frac{E|X^\oplus_{\sigma_{a^{1/q}}}|}{a^{1/q}}\,da\le b+|X|_1\int_b^\infty\frac{1}{a^{1/q}}\,da.
\end{align*}
Taking infimum over $b>0$, we get
\[
E\sup_{t\le T}|X_t|^q\le\frac{1}{1-q}|X|_1^q.
\]
\end{remark}

\begin{definition}
We say that a pair $(Y,Z)$ of $\mathbb{F}$-progressively
measurable processes  is a solution of BSDE with right-hand side
$f+dV$ and terminal value $\xi$ (BSDE($\xi$,$f+dV$) in
abbreviation) if
\begin{enumerate}[{\rm(a)}]
\item $Y$ is a regulated process and $Z\in  \mathcal{H}$,
\item$\int^T_0|f(r,Y_r,Z_r)|\,dr<\infty$,
\item$Y_t=\xi+\int^T_t f(r,Y_r,Z_r)\,dr+\int_t^T\,dV_r
-\int^T_t Z_r\,dB_r$, $t\in[0,T]$.
\end{enumerate}
\end{definition}

The following  definition of a solution of reflected BSDE with one optional barrier was introduced in \cite{KRzS}.

\begin{definition}\label{def1.3}
We say that  a triple $(Y, Z,K)$   of $\mathbb{F}$-progressively
measurable  processes is a solution of the reflected backward
stochastic differential equation with right-hand side $f+dV$,
terminal value $\xi$ and lower barrier $L$
(\underline{R}BSDE($\xi,f+dV,L$) in abbreviation) if
\begin{enumerate}[(a)]
\item $Y$ is a regulated process and $Z\in  \mathcal{H}$,
\item $K\in \mathcal{V}^+$,  $L_t\le Y_t$, $t\in[0,T]$, and
\[
\int^T_0(Y_{r-}-\limsup_{s\uparrow r}L_s)\,dK^*_r
+\sum_{r<T}(Y_r-L_r)\Delta^+K_r=0,
\]
\item $\int^T_0|f(r,Y_r,Z_r)|\,dr<\infty$,
\item $Y_t=\xi+\int^T_t f(r,Y_r,Z_r)\,dr
+\int_t^T\,dK_r+\int_t^T\,dV_r-\int^T_t Z_r\,dB_r,\quad t\in
[0,T]$.
\end{enumerate}
\end{definition}

\begin{definition}\label{r.2}
We say that  a triple $(Y, Z,K)$   of $\mathbb{F}$-progressively
measurable  processes is a solution of the reflected backward
stochastic differential equation with right-hand side $f+dV$,
terminal value $\xi$ and upper barrier $U$
($\mathrm{\overline{R}}$BSDE($\xi,f+dV,U$) in abbreviation) if
$(-Y,-Z,K)$ is a solution of
\underline{R}BSDE($-\xi,-\tilde{f}-dV,-U)$ with
\[
\tilde{f}(t,y,z)=f(t,-y,-z).
\]
\end{definition}

The following  theorem and lemma, which are easy modifications  of
\cite[Theorem 2.10]{KRzS} and \cite[Lemma 2.8]{KRzS},
respectively,  will be used  in Section \ref{sec4}. We omit their
proofs because are the same as the proofs of the corresponding
results from \cite{KRzS}.

\begin{theorem}\label{tw3.2}
Assume that  \textnormal{(H1), (H2), (H4), (H5)} are satisfied,
$(Y^n,Z^n)\in\mathcal{S}\otimes\mathcal{H}$, $D^n\in \mathcal{V},
K^n\in\mathcal{V}^+$, $t\mapsto f(t,Y^n_t,Z^n_t)\in L^1(0,T)$ and
\[
Y^n_t=Y^n_0-\int^t_0 f(s,Y^n_s,Z^n_s)\,ds
-\int^t_0\,dK^n_s+\int^t_0\,dD^n_s+\int^t_0 Z^n_s\,dB_s
\]
for $t\in[0,T]$. Moreover, assume that
\begin{enumerate}[{\rm (a)}]
\item $dD^n\le dD^{n+1}$, $n\in\mathbb{N}$,
$\sup_{n\ge 0}E|D^n|_T<\infty$,

\item $\liminf_{n\rightarrow\infty}
\Big(\int^{\tau}_{\sigma}(Y_s-Y^n_s)\,d(K^n_s-D^n_s)^*
+\sum_{\sigma\le s<\tau}(Y_s-Y^n_s)\Delta^+(K^n_s-D^n_s)\Big)\ge
0$ for any $\sigma,\tau\in\Gamma$ such that $\sigma\le\tau$,

\item there exists a process $C\in\mathcal{V}^{+,1}$ such that
$\Delta^-K^n_t\le\Delta^-C_t$, $t\in[0,T]$,

\item there exist processes $\underline{y},\overline{y}
\in\mathcal{V}^{+,1}+\mathcal{M}_{loc}$ of class \mbox{\rm (D)}
such that
\[
E\int^T_0 f^+(s,\overline{y}_s,0)\,ds+E\int^T_0
f^-(s,\underline{y}_s,0)\,ds<\infty,\quad\overline{y}_t
\le Y^n_t\le\underline{y}_t,\quad t\in[0,T],
\]
\item $E\int^T_0|f(s,0,0)|\,ds<\infty$,
\item $Y^n_t\nearrow Y_t$, $t\in[0,T]$.
\end{enumerate}
Then $Y$ is regulated, $D\in\mathcal{V}^{1}$,  where
$D_t=\lim_{n\rightarrow\infty}D^n_t$, $t\in[0,T]$, and there exist
$K\in\mathcal{V}^+$, $Z\in\mathcal{H}$ such that
\[
Y_t=Y_0-\int^t_0 f(s,Y_s,Z_s)\,ds
-\int^t_0\,dK_s+\int^t_0\,dD_s+\int^t_0 Z_s\,dB_s\quad t\in[0,T],
\]
and
\[
Z^n\rightarrow Z\quad dt\otimes P\mbox{-a.e.}, \quad\int^{T}_0
|f(s,Y^n_s,Z^n_s)-f(s,Y_s,Z_s)|\,ds\rightarrow 0\quad
in\,\,\mbox{probability}\quad P.
\]
Moreover, there exists a chain $\{\tau_k\}\subset\Gamma$ such that for every $p\in(0,2)$,
\begin{equation}
\label{eq2.2} E\int^{\tau_k}_0|Z^n_s-Z_s|^p\,ds\rightarrow 0.
\end{equation}
If $|\Delta^-K_t|=0$, $t\in[0,T]$, then
\mbox{\rm(\ref{eq2.2})} also holds for $p=2$.
\end{theorem}

\begin{lemma}\label{r.5}
Assume that \textnormal{(H1)--(H4)} are satisfied. Let
$D^n,D\in\mathcal{V}$ and  $(Y^n,Z^n)$,
$(Y,Z)\in\mathcal{S}\otimes\mathcal{H}$ be such that $t\mapsto
f(t,Y^n_t,Z^n_t),t\mapsto f(t,Y_t,Z_t)\in L^1(0,T)$ and
\[
Y^n_t=Y^n_0-\int^t_0 f(s,Y^n_s,Z^n_s)\,ds
-\int^t_0\,dD^n_s+\int^t_0 Z^n_s\,dB_s,\quad t\in[0,T],
\]
\[
Y_t=Y_0-\int^t_0 f(s,Y_s,Z_s)\,ds
-\int^t_0\,dD_s+\int^t_0 Z_s\,dB_s,\quad t\in[0,T].
\]
If
\begin{enumerate}[{\rm(a)}]
\item there exists a chain $\{\tau_k\}$, such that
$\sup_{n\ge 0}E((D^n)^+_{\tau_k})^2<\infty$,
\item $\liminf_{n\rightarrow\infty}
(\int^{\tau}_{\sigma}(Y_s-Y^n_s)\,dD^{n,*}_s+\sum_{\sigma\le
s<\tau}(Y_s-Y^n_s)\Delta^+D^n_s)\ge 0$ for all
$\sigma,\tau\in\Gamma$ such that $\sigma\le\tau$,
\item there exists $C\in\mathcal{V}^{+,1}$ such that
$|\Delta^-(Y_t-Y^n_t)|\le|\Delta^-C_t|$, $t\in[0,T]$,
\item there exist processes $\underline{y},
\overline{y}\in\mathcal{V}^{+,1}+\mathcal{M}_{loc}$ of class
\mbox{\rm(D)} such that
\[
\overline{y}_t\le Y^n_t\le\underline{y}_t,\quad t\in[0,T], \quad
E\int^T_0 f^+(s,\overline{y}_s,0)\,ds+E\int^T_0
f^-(s,\underline{y}_s,0)\,ds<\infty,
\]
\item $Y^n_t\rightarrow Y_t$, $t\in[0,T]$,
\end{enumerate}
then
\[
Z^n\rightarrow Z\quad dt\otimes P\mbox{-a.e.}, \quad\int^{T}_0
|f(s,Y^n_s,Z^n_s)-f(s,Y_s,Z_s)|\,ds\rightarrow0\quad
in\,\,\mbox{probability}\quad P
\]
and there exists a chain $\{\tau_k\}\subset\Gamma$ such that
for all  $k\in\mathbb{N}$ and $p\in(0,2)$,
\begin{equation}
\label{eq2.1} E\int^{\tau_k}_0|Z^n_s-Z_s|^p\,ds\rightarrow0.
\end{equation}
If $\Delta^-C_t=0$, $t\in[0,T]$, then \mbox{\rm(\ref{eq2.1})} also
holds for $p=2$.
\end{lemma}

\nsubsection{Existence and uniqueness of solutions}

\label{sec3}
\subsubsection{Definition of a solution and comparison results }
\begin{definition}
\label{def1} We say that  a triple $(Y, Z,K)$   of
$\mathbb{F}$-progressively measurable  processes is a solution of
the reflected backward stochastic differential equation with
right-hand side $f+dV$, terminal value $\xi$, lower barrier $L$ and upper barrier $U$ (\textnormal{RBSDE($\xi,f+dV,L,U$)}) if
\begin{enumerate}[(a)]
\item $Y$ is a regulated process and $Z\in\mathcal{H}$,
\item $R\in \mathcal{V}$,  $L_t\le Y_t\le U_t$, $t\in[0,T]$, and
\begin{align*}
\int_0^T (Y_{r-}&-\limsup_{s\uparrow r}L_{s})\,dR^{+,*}_r+\sum_{r<T}(Y_r-L_r)\Delta^+R^+_r\\
&+\int_0^T (\liminf_{s\uparrow r}U_{s}-Y_{r-})\,dR^{-,*}_r+\sum_{r<T}(U_r-Y_r)\Delta^+R^-_r=0,
\end{align*}
where $R=R^+-R^-$ is the Jordan decomposition of $R$,
\item $\int^T_0|f(r,Y_r,Z_r)|\,dr<\infty$,
\item $Y_t=\xi+\int^T_t f(r,Y_r,Z_r)\,dr
+\int_t^T\,dV_r+\int_t^T\,dR_r-\int^T_t Z_r\,dB_r,\quad t\in
[0,T]$.
\end{enumerate}
\end{definition}

Note that if $L,U$ are regulated processes  and $(Y,Z,R)$ is a
solution to {\rm RBSDE}$(\xi,f,L,U)$ then
\[
\Delta^-R^+_t=(Y_t-L_{t-}+\Delta^-V_t)^-,\qquad
\Delta^-R^-_t=(Y_t-U_{t-}+\Delta^-V_t)^+,
\]
and
\[
\Delta^+R^+_t=(Y_{t+}-L_{t}+\Delta^+V_t)^-,\qquad
\Delta^+R^-_t=(Y_{t+}-U_{t}+\Delta^+V_t)^+.
\]
To check the first equality (the proofs of the other  ones are
similar) assume first that $\Delta^-R^+_t>0$ and observe that by
Definition \ref{def1}(d),
\[
\Delta^-R^+_t=-\Delta^-Y_t-\Delta^-V_t.
\]
Since by Definition \ref{def1}(b), $Y_{t-}=L_{t-}$,  the desired
equality holds true. Now assume that $\Delta^-R^+_t=0$.  Since
$\Delta^-R^-_t\geq0$  and $Y_{t-}\geq L_{t-}$,
\[
Y_t-L_{t-}+\Delta^-V_t=\Delta^-Y_t+Y_{t-}-L_{t-}+\Delta^-V_t
=-\Delta R^+_t+\Delta^-R^-_{t}+Y_{t-}-L_{t-}\geq0,
\]
which completes the proof.

From the above equalities it follows  in  particular that if the
barriers and $V$ are c\`adl\`ag (resp. c\`agl\`ad), then $Y$ is
c\`adl\`ag (resp. c\`agl\`ad).

\begin{proposition}\label{proposition2.2}
Let $(Y^i,Z^i,R^i)$ be a solution of
\textnormal{RBSDE($\xi^i,f^i+dV^i,L^i,U^i$)}, $i=1,2$. Assume that
$f^1$ satisfies \mbox{\rm (H1), (H2)} and $\xi^1\le\xi^2$,
$f^1(\cdot,Y^2,Z^2)\le f^2(\cdot,Y^2,Z^2)$ $dt\otimes dP$-a.s.,
$dV^1\le dV^2$, $L^1\le L^2$, $U^1\le U^2$. If
$(Y^1-Y^2)^+\in\mathcal{S}^p$ for some $p>1$, then $Y^1\le Y^2$.
\end{proposition}

\begin{proof}
Without loss of generality we may assume that
$\mu=-\frac{4\lambda^2}{p-1}$ (see \cite[Remark 3.2]{KRzS}). By
(H1), (H2) and the fact that $f^1(\cdot,Y^2,Z^2)\le
f^2(\cdot,Y^2,Z^2)$ $dt\otimes dP$-a.s., we have
\begin{align}
\nonumber\label{min2.2.1.0}
&((Y^1_r-Y^2_r)^+)^{p-1}(f^1(r,Y^1_r,Z^1_r)-f^2(r,Y^2_r,Z^2_r))\\
&\quad\le((Y^1_r-Y^2_r)^+)^{p-1}(f^1(r,Y^1_r,Z^1_r)-f^1(r,Y^2_r,Z^2_r))\nonumber\\
&\quad\le-\frac{4\lambda^2}{p-1}((Y^1_r-Y^2_r)^+)^p
+\lambda((Y^1_r-Y^2_r)^+)^{p-1}|Z^1_r-Z^2_r|.
\end{align}
Note that, by the minimality condition for $R^1,R^2$ and the
assumption that $L^1\le L^2$ and $U^1\le U^2$,
\begin{align}\label{min2.2.1}
\mathbf{1}_{\{Y^1_{r-}>Y^2_{r-}\}}\,d(R^1_r-R^2_r)^*\le \mathbf{1}_{\{Y^1_{r-}>Y^2_{r-}\}}\,dR^{1,*,+}_r+\mathbf{1}_{\{Y^1_{r-}>Y^2_{r-}\}}\,dR^{2,*,-}_r=0,
\end{align}
and
\begin{align}\label{min2.2.2}
\mathbf{1}_{\{Y^1_r>Y^2_r\}}\Delta^+(R^1_r-R^2_r)\le\mathbf{1}_{\{Y^1_r>Y^2_r\}}\Delta^+R^{1,+}_r+\mathbf{1}_{\{Y^1_r>Y^2_r\}}\Delta^+R^{1,-}_r=0.
\end{align}
By \cite[Corollary A.5]{KRzS}, for $\tau,\sigma\in\Gamma$ such
that $\sigma\le\tau$ we have
\begin{align*}
&((Y^1_{\sigma}-Y^2_{\sigma})^+)^p
+\frac{p(p-1)}{2}\int^{\tau}_{\sigma}((Y^1_r-Y^2_r)^+)^{p-2}
\mathbf{1}_{\{Y^1_r>Y^2_r\}}|Z^1_r-Z^2_r|^2\,dr\\
&\quad\le((Y^1_{\tau}-Y^2_{\tau})^+)^p+p\int^{\tau}_{\sigma}
((Y^1_r-Y^2_r)^+)^{p-1}(f^1(r,Y^1_r,Z^1_r)-f^2(r,Y^2_r,Z^2_r))\,dr\\
&\quad\quad+p\int^{\tau}_{\sigma}((Y^1_{r-}-Y^2_{r-})^+)^{p-1}
\,d(V^1_r-V^2_r)^*+p\sum_{\sigma\le r<\tau}((Y^1_r-Y^2_r)^+)^{p-1}
\Delta^+(V^1_r-V^2_r)\\
&\quad\quad+p\int^{\tau}_{\sigma}((Y^1_{r-}-Y^2_{r-})^+)^{p-1}\mathbf{1}_{\{Y^1_r>Y^2_r\}}\,d(R^1_r-R^2_r)^*\\
&\quad\quad+p\sum_{\sigma\le r<\tau}((Y^1_r-Y^2_r)^+)^{p-1}\mathbf{1}_{\{Y^1_r>Y^2_r\}}\Delta^+(R^1_r-R^2_r)\\
&\quad\quad-p\int^{\tau}_{\sigma}((Y^1_r-Y^2_r)^+)^{p-1}(Z^1_r-Z^2_r)\,dB_r.
\end{align*}
By the above inequality, (\ref{min2.2.1.0})--(\ref{min2.2.2}) and
the assumption that $dV^1\le dV^2$, we get
\begin{align}
\label{eq2.2.0}
&((Y^1_{\sigma}-Y^2_{\sigma})^+)^p+\frac{p(p-1)}{2}
\int^{\tau}_{\sigma}((Y^1_r-Y^2_r)^+)^{p-2}
\mathbf{1}_{\{Y^1_r>Y^2_r\}}|Z^1_r-Z^2_r|^2\,dr\nonumber\\
&\quad\le((Y^1_{\tau}-Y^2_{\tau})^+)^p
-\frac{4p\lambda^2}{p-1}\int^{\tau}_{\sigma}((Y^1_r-Y^2_r)^+)^p\,dr
+p\lambda\int^{\tau}_{\sigma}
((Y^1_r-Y^2_r)^+)^{p-1}|Z^1_r-Z^2_r|\,dr\nonumber\\
&\quad\quad-p\int^{\tau}_{\sigma}((Y^1_r-Y^2_r)^+)^{p-1}
(Z^1_r-Z^2_r)\,dB_r.
\end{align}
Note that
\begin{align*}
&p\lambda((Y^1_r-Y^2_r)^+)^{p-1}|Z^1_r-Z^2_r|\\
&\quad=p((Y^1_r-Y^2_r)^+)^{p-2}\mathbf{1}_{\{Y^1_r>Y^2_r\}}
(\lambda(Y^1_r-Y^2_r)^+|Z^1_r-Z^2_r|)\\
&\quad\le p((Y^1_r-Y^2_r)^+)^{p-2}\mathbf{1}_{\{Y^1_r>Y^2_r\}}
\Big(\frac{4\lambda^2}{p-1}((Y^1_r-Y^2_r)^+)^2+\frac{p-1}{4}
|Z^1_r-Z^2_r|^2\Big)\\
&\quad\le\frac{4p\lambda^2}{p-1}((Y^1_r-Y^2_r)^+)^p
+\frac{p(p-1)}{4}((Y^1_r-Y^2_r)^+)^{p-2}
\mathbf{1}_{\{Y^1_r>Y^2_r\}}|Z^1_r-Z^2_r|^2.
\end{align*}
From this and (\ref{eq2.2.0}) it follows that
\begin{align}
\label{eq2}
&((Y^1_{\sigma}-Y^2_{\sigma})^+)^p+\frac{p(p-1)}{4}
\int^\tau_{\sigma}((Y^1_r-Y^2_r)^+)^{p-2}
\mathbf{1}_{\{Y^1_r>Y^2_r\}}|Z^1_r-Z^2_r|^2\,dr\nonumber\\
&\quad\le((Y^1_{\tau}-Y^2_{\tau})^+)^p-p\int^\tau_{\sigma}
((Y^1_r-Y^2_r)^+)^{p-1}(Z^1_r-Z^2_r)\,dB_r.
\end{align}
Let $\{\tau_k\}\subset\Gamma$ be a localizing sequence for the
local martingale $\int^{\cdot}_{\sigma}
((Y^1_r-Y^2_r)^+)^{p-1}(Z^1_r-Z^2_r)\,dB_r$. By (\ref{eq2})  with
$\tau$ replaced by $\tau_k\ge\sigma$, we have
\[
((Y^1_{\sigma}-Y^2_{\sigma})^+)^p
\leq((Y^1_{\tau_k}-Y^2_{\tau_k})^+)^p-p\int^{\tau_k}_{\sigma}
((Y^1_r-Y^2_r)^+)^{p-1}(Z^1_r-Z^2_r)\,dB_r,\quad k\in\N.
\]
Taking the expectation and then letting $k\rightarrow\infty$, we
get $E((Y^1_{\sigma}-Y^2_{\sigma})^+)^p=E((\xi^1-\xi^2)^+)^p=0$.
Hence, by the Section Theorem (see, e.g., \cite[Chapter IV,
Theorem 86]{dm}), $(Y^1_t-Y^2_t)^+=0$, $t\in[0,T]$.
\end{proof}

\begin{lemma}\label{uw2.0} Let $(Y^i,Z^i,R^i)$ be
a solution of \textnormal{RBSDE($\xi^i,f^i+dV^i,L^i,U^i)$},
$i=1,2$.  Assume that $f^1$ satisfies \mbox{\rm(H2),(Z)}, $Y^1,Y^2$ are of class \mbox{\rm{(D)}} and
$Z^1,Z^2\in L^q((0,T)\otimes\Omega)$ for some $q\in(\alpha,1]$.
Assume also that  $\xi^1\le\xi^2$, $f^1(\cdot,Y^2,Z^2)\le
f^2(\cdot,Y^2,Z^2)$ $dt\otimes dP$-a.s., $dV^1\le dV^2$, $L^1\le
L^2$, $U^1\le U^2$. Then $(Y^1-Y^2)^+\in\mathcal{S}^p$ for
$p=\frac{q}{\alpha}$.
\end{lemma}
\begin{proof}
By \cite[Corollary A.5]{KRzS}, the assumptions, (H2),
(\ref{min2.2.1}) and (\ref{min2.2.2}), for all $\sigma,\tau\in\Gamma$,
such that $\sigma\le \tau$ we have
\begin{align}\label{eq2.2.0.0}
\nonumber
(Y^1_{\sigma}-Y^2_{\sigma})^+&\le(Y^1_{\tau}-Y^2_{\tau})^+
+\int^{\tau}_{\sigma}\mathbf{1}_{\{Y^1_r>Y^2_r\}}(f^1(r,Y^1_r,Z^1_r)
-f^2(r,Y^2_r,Z^2_r))\,dr\\ \nonumber
&\quad+\int^{\tau}_{\sigma}\mathbf{1}_{\{Y^1_{r-}>Y^2_{r-}\}}\,d(V^1_r-V^2_r)^*
+\sum_{\sigma\le r<\tau}\mathbf{1}_{\{Y^1_r>Y^2_r\}}\Delta^+(V^1_r-V^2_r)\\ \nonumber
&\quad+\int^{\tau}_{\sigma}\mathbf{1}_{\{Y^1_{r-}>Y^2_{r-}\}}\,d(R^1_r-R^2_r)^*
+\sum_{\sigma\le r<\tau}\mathbf{1}_{\{Y^1_r>Y^2_r\}}\Delta^+(R^1_r-R^2_r)\\ \nonumber
&\quad-\int^{\tau}_{\sigma}\mathbf{1}_{\{Y^1_r>Y^2_r\}}(Z^1_r-Z^2_r)\,dB_r\\ \nonumber
&\le\int^{\tau}_{\sigma}\mathbf{1}_{\{Y^1_r>Y^2_r\}}(f^1(r,Y^2_r,Z^1_r)
-f^1(r,Y^2_r,Z^2_r))\,dr\\
&\quad-\int^{\tau}_{\sigma}\mathbf{1}_{\{Y^1_r>Y^2_r\}}(Z^1_r-Z^2_r)\,dB_r.
\end{align}
By (Z),
\begin{align*}
|f^1(r,Y^2_r,Z^1_r)-f^1(r,Y^2_r,Z^2_r)|
&\le|f^1(r,Y^2_r,Z^1_r)-f^1(r,Y^2_r,0)|\\
&\quad+|f^1(r,Y^2_r,0)-f^1(r,Y^2_r,Z^2_r)|\\
&\le2\gamma(g_r+|Y^1_r|+|Y^2_r|+|Z^1_r|+|Z^2_r|)^{\alpha}.
\end{align*}
Let $\{\tau_k\}$ be a localizing sequence for the local martingale $\int^{\cdot}_{\sigma}\mathbf{1}_{\{Y^1_r>Y^2_r\}}(Z^1_r-Z^2_r)\,dB_r$. From the above inequality and (\ref{eq2.2.0.0}) we get
\begin{align*}
(Y^1_{\sigma}-Y^2_{\sigma})^+\le E\Big((Y^1_{\tau_k}-Y^2_{\tau_k})^++2\gamma
\int^T_0(g_r+|Y^1_r|+|Y^2_r|+|Z^1_r|
+|Z^2_r|)^{\alpha}\,dr|\mathcal{F}_{\sigma}\Big).
\end{align*}
Since $Y^1,Y^2$ are of class (D), $\{\tau_k\}$ is a chain and
$\xi^1\le\xi^2$, letting $k\rightarrow\infty$ in the above
inequality we get
\begin{align*}
(Y^1_{\sigma}-Y^2_{\sigma})^+\le 2\gamma E\Big(\int^T_0(g_r+|Y^1_r|+|Y^2_r|+|Z^1_r|
+|Z^2_r|)^{\alpha}\,dr|\mathcal{F}_{\sigma}\Big).
\end{align*}
Let $p=q/\alpha$. By Doob's inequality,
\begin{align*}
E\sup_{t\le T}((Y^1_t-Y^2_t)^+)^p
\le C_pE\Big(\int^T_0(g_r+|Y^1_r|+|Y^2_r|+|Z^1_r|+|Z^2_r|)^q\,dr\Big).
\end{align*}
Hence $(Y^1-Y^2)^+\in\mathcal{S}^p$.
\end{proof}

\begin{remark}
\label{uwaga2.3} Observe that if $f^1,\, f^2$ do not depend  on
$z$, then in Proposition \ref{proposition2.2} it  is enough to
assume that $(Y^1-Y^2)^+$ is of class \textnormal{(D)}.
\end{remark}

\begin{proposition}\label{prop2.5}
Let $(Y,Z,R)$ be a solution of \textnormal{RBSDE}($\xi,f+dV,L,U$).
Assume that $p>1$, \textnormal{(H1), (H2), (H3)} are satisfied,
$Y\in\mathcal{S}^p$, $Z\in\mathcal{H}^p$, $R\in\mathcal{V}^p$ or
$p=1$, \textnormal{(H1), (H2), (H3), (Z)} are satsfied, $Y$ is of
class \textnormal{(D)}, $Z\in\mathcal{H}^q$, $q\in(0,1)$ and
$R\in\mathcal{V}^1$. Then,
\begin{equation}\label{prop2.5.0}
E\Big(\int^T_0|f(r,Y_r,Z_r)|\,dr\Big)^p<\infty.
\end{equation}
\end{proposition}
\begin{proof}
We may assume that $\mu=0$. By \cite[Corollary A.5]{KRzS}, for all $\sigma,\tau\in\Gamma$ such
that $\sigma\le\tau$, we have
\begin{align}\label{prop2.5.1}
\nonumber
|Y_{\sigma}|&\le|Y_{\tau}|+\int^{\tau}_{\sigma}\sgn(Y_r)f(r,Y_r,Z_r)\,dr+\int^{\tau}_{\sigma}\sgn(Y_{r-})\,dV_r+\int^{\tau}_{\sigma}\sgn(Y_{r-})\,dR^*_r\\
&\quad+\sum_{\sigma\le r<\tau}\sgn(Y_r)\Delta^+R_r-\int^{\tau}_{\sigma}\sgn(Y_r)Z_r\,dB_r
\end{align}
By (Z) and (H2),
\begin{equation}\label{prop2.5.2}
\sgn(Y_r)f(r,Y_r,Z_r)\le-|f(r,Y_r,Z_r)|
+2\gamma(g_r+|Z_r|)^{\alpha}+2|f(r,0,0)|,
\end{equation}
whereas by (H1) and (H2),
\begin{equation}\label{prop2.5.3}
\sgn(Y_r)f(r,Y_r,Z_r)\le-|f(r,Y_r,Z_r)|+2\lambda|Z_r|+2|f(r,0,0)|.
\end{equation}
From  (\ref{prop2.5.1})--(\ref{prop2.5.3}) and the assumptions we
get the desired result.
\end{proof}

\subsubsection{Existence of solutions for  $f$ independent of $z$}

\begin{theorem}\label{thmz}
Assume that $f$ is independent of  $z$. If $p>1$  and
\textnormal{(H1)--(H6)} (resp. $p=1$ and \textnormal{(H1)--(H5),
(H6*)}) are satisfied, then there exists a unique solution
$(Y,Z,R)$ of \textnormal{RBSDE}($\xi$,$f+dV$,$L$,$U$) such that
$Y\in\mathcal{S}^p$, $Z\in\mathcal{H}^p$ and $R\in\mathcal{V}^p$
(resp. $Y$ is of class \textnormal{(D)}, $Z\in\mathcal{H}^q$,
$q\in(0,1)$, and $R\in\mathcal{V}^1$).
\end{theorem}
\begin{proof}

Without loss of generality we may assume that $\mu=0$.  Let
$(Y^{1,0},Z^{1,0})$ be a solution of BSDE($\xi$,$f+dV$) such that
if $p>1$, $Y^{1,0}\in\mathcal{S}^p$, $Z^{1,0}\in\mathcal{H}^p$ and
if $p=1$, $Y^{1,0}$ is of class (D), $Z^{1,0}\in\mathcal{H}^q$,
$q\in(0,1)$. Set $(Y^{2,0},Z^{2,0})=(0,0)$. Moreover,
for each $n\ge 1$ let $(Y^{1,n},Z^{1,n},K^{1,n})$ be a solution of
\underline{R}BSDE($\xi$,$f_n+dV$,$L+Y^{2,n-1}$) with
\[
f_n(r,y)=f(r,y-Y^{2,n-1}_r),
\]
and let $(Y^{2,n}, Z^{2,n},K^{2,n})$ be a solution of \underline{R}BSDE($0$,$0$,$Y^{1,n-1}-U$) such that if $p>1$, $Y^{1,n},Y^{2,n}\in\mathcal{S}^p$, $Z^{1,n},Z^{2,n}\in\mathcal{H}^p$, $K^{1,n},K^{2,n}\in\mathcal{V}^{+,p}$, and if $p=1$, then $Y^{1,n},Y^{2,n}$ are of class (D), $Z^{1,n},Z^{2,n}\in\mathcal{H}^q$, $q\in(0,1)$, $K^{1,n},K^{2,n}\in\mathcal{V}^{+,1}$. For each $n\ge 0$ the existence of the above solutions follows from \cite[Theorem 3.20]{KRzS}. In both cases ($p>1$, $p=1$), by Proposition \ref{prop2.5}, we have
\begin{equation}\label{th1.0.0}
E\Big(\int^T_0|f(r,Y^{1,n}_r-Y^{2,n-1}_r)|\,dr\Big)^p<\infty.
\end{equation}
The rest of the proof we divide into 4 steps.

{\bf Step 1.} We show that the sequences $(Y^{1,n})_{n\ge 0}$, $(Y^{2,n})_{n\ge 0}$ are increasing. We proceed by induction. Clearly  $Y^{1,1}\ge Y^{1,0}$ and $Y^{2,1}\ge Y^{2,0}$. Suppose that  $Y^{1,n}\ge Y^{1,n-1}$ and $Y^{2,n}\ge Y^{2,n-1}$. Using (H2) we show  that $f_{n+1}\ge f_n$ and $L+Y^{2,n}\ge L+Y^{2,n-1}$. Hence, by Proposition \ref{proposition2.2} and Remark \ref{uwaga2.3}, $Y^{1,n+1}\ge Y^{1,n}$. By a similar argument, $Y^{2,n+1}\ge Y^{2,n}$, so $(Y^{1,n})_{n\ge 0}$, $(Y^{2,n})_{n\ge 0}$ are increasing.

{\bf Step 2.} Let $Y^1:=\sup_{n\ge 1}Y^{1,n}$, $Y^2:=\sup_{n\ge 1}Y^{2,n}$. We show that $Y^1,Y^2\in \mathcal{S}^p$ if $p>1$, and if $p=1$, then $Y^1,Y^2$ are of class (D). Let $p>1$. By (H6), there exists a
process $X\in(\mathcal{M}_{loc} +\mathcal{V}^p)\cap\mathcal{S}^p$ such that $X \ge L$ and $\int^T_0 f^-(r,X_r, 0)\,dr\in \mathbb{L}^p$. If $p = 1$,
then by (H6*), there exists $X$ of class (D) such that $X\in \mathcal{M}_{loc}+\mathcal{V}^1$, $X \ge L$ and $\int^T_0 f^-(r,X_r, 0)\,dr\in \mathbb{L}^1$. Since the Brownian filtration has the representation property, there exist processes $H\in\mathcal{M}_{loc}$ and $C\in\mathcal{V}^p$ such that
\[
X_t=X_T-\int^T_t\,dC_r-\int^T_t H_r\,dB_r,\quad t\in[0,T].
\]
This equality can be written in the form
\[
\tilde{X}_t=\xi+\int^T_t f(r,\tilde{X}_r)\,dr+\int^T_t\,dV_r+\int^T_t\,dC'_r-\int^T_t H_r\,dB_r,
\]
where $C'$ is some process in $\mathcal{V}^p$, $\tilde{X}_t=X_t$, $t\in[0,T)$, $\tilde{X}_T=\xi$.
Let $(\tilde{X}^1,\tilde{H}^1)$ be a solution of the following BSDE
\[
\tilde{X}^1_t=\xi+\int^T_t f(r,\tilde{X}_r)\,dr+\int^T_t\,dV_r+\int^T_t\,dC'^+_r-\int^T_t \tilde{H}^1_r\,dB_r,\quad t\in[0,T],
\]
and $(\tilde{X}^2,\tilde{H}^2)$ be a solution of the BSDE
\[
\tilde{X}^2_t=\int^T_t\,dC'^-_r-\int^T_t \tilde{H}^2_r\,dB_r,\quad t\in[0,T],
\]
such that if $p>1$, then $\tilde{X}^1,\tilde{X}^2\in\mathcal{S}^p$, $\tilde{H}^1,\tilde{H}^2\in\mathcal{H}^p$, and if $p=1$, then  $\tilde{X}^1,\tilde{X}^2$ are of class (D), $\tilde{H}^1,\tilde{H}^2\in\mathcal{H}^q$, $q\in(0,1)$. The existence of such  solutions follows from \cite[Theorem 3.20]{KRzS}. Let us note that $\tilde{X}=\tilde{X}^1-\tilde{X}^2$. It is easy to see that $(\tilde{X}^1,\tilde{H}^1,0)$ is a solution of \underline{R}BSDE($\xi$, $\tilde{f}+dV+dC'^+$, $L+\tilde{X}^2$) with $\tilde{f}(r,x)=f(r,x-\tilde{X}^2_r)$ and $(\tilde{X}^2,\tilde{H}^2,0)$ is a solution of \underline{R}BSDE($0$, $dC'^-$, $\tilde{X}^1-U$). Proceeding by induction we will show that for each $n\in\mathbb{N}$, $\tilde{X}^1\ge Y^{1,n}$ and $\tilde{X}^2\ge Y^{2,n}$. For $n=0$, since $\tilde{X}^2\ge 0$, using (H2) we get  $\tilde{f}\ge f$. Hence, by Proposition \ref{proposition2.2} and Remark \ref{uwaga2.3}, $\tilde{X}^1\ge Y^{1,0}$. It is clear that
$\tilde{X}^2\ge Y^{2,0}$ since $Y^{2,0}=0$. Suppose that $\tilde{X}^1\ge Y^{1,n}$ and $\tilde{X}^2\ge Y^{2,n}$. Using (H2) we show that $\tilde{f}\ge f_{n+1}$, $L+\tilde{X}^2\ge L+Y^{2,n}$, $Y^{1,n}-U\le\tilde{X}-U$. Hence by Proposition \ref{proposition2.2} and Remark \ref{uwaga2.3}, $\tilde{X}^2\ge Y^{1,n+1}$, $\tilde{X}^2\ge Y^{2,n+1}$, so for each $n\in\mathbb{N}$, $\tilde{X}^1\ge Y^{1,n}$ and $\tilde{X}^2\ge Y^{2,n}$. We have
\begin{equation}\label{th1.0}
Y^{1,0}\le Y^{1,n}\le \tilde{X}^1,\quad Y^{2,0}\le Y^{2,n}\le \tilde{X}^2.
\end{equation}
Therefore $Y^1,Y^2\in \mathcal{S}^p$ for $p>1$, and if $p=1$, then $Y^1,Y^2$ are of class (D).

{\bf Step 3.} We will show that there exist $Z^1,Z^2\in\mathcal{H}^p$, $K^1,K^2\in\mathcal{V}^p$ if $p>1$,  and $Z^1,Z^2\in\mathcal{H}^q$, $q\in(0,1)$, $K^1,K^2\in\mathcal{V}^1$ if $p=1$ such that $(Y^1,Z^1,K^1)$ is a solution of \underline{R}BSDE($\xi$,$\hat{f}+dV$,$L+Y^2$) with $\hat{f}(r,y)=f(r,y-Y^2_r)$
and $(Y^2,Z^2,K^2)$ is a solution of \underline{R}BSDE($0$,$0$,$Y^1-U$). By (H4), $f(r,Y^{1,n}_r-Y^{2,n-1}_r)\rightarrow f(r,Y^1_r-Y^2_r)$ as $n\rightarrow\infty$. Furthermore, by (H2) and (\ref{th1.0}),
\begin{equation}\label{th1.1}
f(r,X^1_r)\le f(r,Y^{1,n}_r-Y^{2,n-1}_r)\le f(r,Y^{1,0}_r-X^2_r).
\end{equation}
Hence,  by (H2), (H5) and the Lebesgue dominated convergence theorem,
\begin{equation}
\int^T_0 |f(r,Y^{1,n}_r-Y^{2,n-1}_r)\,dr-f(r,Y^1_r-Y^2_r)|\,dr\rightarrow 0.
\end{equation}
Observe that
\begin{align*}
S^n_t:&=Y^{1,n}_t+\int^t_0 f(r,Y^{1,n}_r-Y^{2,n-1}_r)\,dr+V_t\\
&\ge L_t+Y^{2,n-1}_t+\int^t_0 f(r,Y^{1,n}_r-Y^{2,n-1}_r)\,dr+V_t=:\bar{L}^n_t,\quad t \in[0,T],
\end{align*}
and that (\ref{th1.0.0}) implies that $S^n$ is a supermartingale of class (D) on $[0,T]$. Letting $n\rightarrow\infty$ in the above inequality we get
\[
S_t:=Y^1_t+\int^t_0 f(r,Y^1_r-Y^2_r)\,dr+V_t\ge L_t+Y^2_t+\int^t_0 f(r,Y^1_r-Y^2_r)\,dr+V_t=:\bar{L}_t,\quad t\in[0,T].
\]
Set $\tau_k=\inf\{t\ge 0;\,\int^t_0|f(r,Y^{1,0}_r-X^2_r)|+|f(r,X^1_r)|\,dr\ge k\}\wedge T$, $k\in\N$.
By (H5), $\{\tau_k\}$ is a chain. From the definition $\{\tau_k\}$, and (\ref{th1.0}), (\ref{th1.1}) it follows that  $S$ is a supermartingale of class (D) on $[0,\tau_k]$, $k\ge 0$. It is clear that $S$ majorizes $\bar{L}$, so for $\sigma\in\Gamma$ we have
\[
S_{\sigma\wedge\tau_k}\ge\esssup_{\tau\in\Gamma_{\sigma\wedge\tau_k}}E(\bar{L}_{\tau\wedge\tau_k}|\mathcal{F}_{\sigma\wedge\tau_k}).
\]
Hence
\begin{align}
\nonumber\label{tw1.1}
Y^1_{\sigma\wedge\tau_k}&\ge\esssup_{\tau\in\Gamma_{\sigma\wedge\tau_k}}E\Big(\int^{\tau\wedge\tau_k}_{\sigma\wedge\tau_k} f(r,Y^1_r-Y^2_r)\,dr+\int^{\tau\wedge\tau_k}_{\sigma\wedge\tau_k}\,dV_r+(L_{\tau}+Y^2_{\tau})\mathbf{1}_{\{\tau<\tau_k\}}\\
&\qquad\qquad\qquad\qquad+Y^1_{\tau_k}\mathbf{1}_{\{\tau=\tau_k\}}|\mathcal{F}_{\sigma\wedge\tau_k}\Big).
\end{align}
To show the opposite inequality, we first note that the triple  $(Y^{1,n},Z^{1,n},K^{1,n})$ is a solution of \underline{R}BSDE($Y^{1,n}_{\tau_k}$,$f_n+dV$,$L+Y^{2,n-1}$) on $[0,\tau_k]$, so by \cite[Proposition 3.13]{KRzS},  for $\sigma\in\Gamma$ we have
\begin{align}\label{tw1.1.1}
\nonumber
Y^{1,n}_{\sigma\wedge\tau_k}&=\esssup_{\tau\in\Gamma_{\sigma\wedge\tau_k}}E\Big(\int^{\tau\wedge\tau_k}_{\sigma\wedge\tau_k} f(r,Y^{1,n}_r-Y^{2,n-1}_r)\,dr+\int^{\tau\wedge\tau_k}_{\sigma\wedge\tau_k}\,dV_r\\\nonumber&\qquad\qquad\qquad+(L_{\tau}+Y^{2,n-1}_{\tau})\mathbf{1}_{\{\tau<\tau_k\}}
+Y^{1,n}_{\tau_k}\mathbf{1}_{\{\tau=\tau_k\}}|\mathcal{F}_{\sigma\wedge\tau_k}\Big)\\\nonumber&\quad \le\esssup_{\tau\in\Gamma_{\sigma\wedge\tau_k}}E\Big(\int^{\tau\wedge\tau_k}_{\sigma\wedge\tau_k} f(r,Y^{1,n}_r-Y^{2,n-1}_r)\,dr+\int^{\tau\wedge\tau_k}_{\sigma\wedge\tau_k}\,dV_r\\
&\qquad\qquad\qquad\qquad+(L_{\tau}+Y^2_{\tau})\mathbf{1}_{\{\tau<\tau_k\}}+Y^1_{\tau_k}\mathbf{1}_{\{\tau=\tau_k\}}|\mathcal{F}_{\sigma\wedge\tau_k}\Big).
\end{align}
By the definition of $\tau_k$ and (\ref{th1.1}),
\begin{equation}\label{tw1.1.2}
E\int^{\tau_k}_0|f(r,Y^{1,n}_r-Y^{2,n-1}_r)-f(r,Y^1_r-Y^2_r)|\,dr\rightarrow 0.
\end{equation}
By (\ref{tw1.1.1}), (\ref{tw1.1.2}) and \cite[Lemma 3.19]{KRzS},
\begin{equation*}
Y^1_{\sigma\wedge\tau_k}\le\esssup_{\tau\in\Gamma_{\sigma\wedge\tau_k}}E\Big(\int^{\tau\wedge\tau_k}_{\sigma\wedge\tau_k} f(r,Y^1_r-Y^2_r)\,dr+\int^{\tau\wedge\tau_k}_{\sigma\wedge\tau_k}\,dV_r+(L_{\tau}+Y^2_{\tau})\mathbf{1}_{\{\tau<\tau_k\}}
+Y^1_{\tau_k}\mathbf{1}_{\{\tau=\tau_k\}}|\mathcal{F}_{\sigma\wedge\tau_k}\Big).
\end{equation*}
By the above inequality and (\ref{tw1.1}),
\begin{align}
\nonumber
Y^1_{\sigma\wedge\tau_k}&=\esssup_{\tau\in\Gamma_{\sigma\wedge\tau_k}}E\Big(\int^{\tau\wedge\tau_k}_{\sigma\wedge\tau_k} f(r,Y^1_r-Y^2_r)\,dr+\int^{\tau\wedge\tau_k}_{\sigma\wedge\tau_k}\,dV_r+(L_{\tau}+Y^2_{\tau})\mathbf{1}_{\{\tau<\tau_k\}}\\
&\qquad\qquad\qquad+Y^1_{\tau_k}\mathbf{1}_{\{\tau=\tau_k\}}|\mathcal{F}_{\sigma\wedge\tau_k}\Big).
\end{align}
On the other hand, by \cite[Proposition 3.13]{KRzS}, for every $\sigma\in\Gamma$,
%letting $n\rightarrow\infty$ and using
\[
Y^{2,n}_{\sigma}=\esssup_{\tau\in\Gamma_{\sigma}}E\Big((Y^{1,n-1}_{\tau}-U_{\tau})\mathbf{1}_{\{\tau<T\}}|\mathcal{F}_{\sigma}\Big).
\]
Since $\{Y^{1,n}\},\{Y^{2,n}\}$ are nondecreasing, letting $n\rightarrow\infty$ and using  standard properties of the Snell envelope we obtain
\[
Y^2_{\sigma}=\esssup_{\tau\in\Gamma_{\sigma}}E\Big((Y^1_{\tau}-U_{\tau})\mathbf{1}_{\{\tau<T\}}|\mathcal{F}_{\sigma}\Big).
\]
We have showed that $S$ is a supermartingale on $[0,\tau_k]$, so by the Mertens decomposition there exist $K^{1,k}\in\mathcal{V}^{1,+},Z^{1,k}\in\mathcal{H}$ such that
\[
Y^1_t=Y^1_{\tau_k}+\int^{\tau_k}_t f(r,Y^1_r-Y^2_r)\,dr+\int^{\tau_k}_t\,dV_r+\int^{\tau_k}_t\,dK^{1,k}_r-\int^{\tau_k}_t Z^{1,k}_r\,dB_r,\quad t\in[0,\tau_k].
\]
By  \cite[Corollary 3.11]{KRzS},
\[
\int^{\tau_k}_0(Y^1_{r-}-\limsup_{s\uparrow r}(L_s+Y^2_s))\,dK^{1,k,*}_r
+\sum_{r<\tau_k}(Y^1_r-(L_r+Y^2))\Delta^+K^{1,k}_r=0.
\]
Therefore $(Y^1,Z^{1,k},K^{1,k})$ is a solution of \underline{R}BSDE($Y^1_{\tau_k},\hat{f}+dV,L+Y^2$) on $[0,\tau_k]$. By uniqueness, $K^{1,k}_t=K^{1,k+1}_t$ and  $Z^{1,k}_t=Z^{1,k+1}_t$ for $t\in[0,\tau_k]$, so using the fact that $\{\tau_k\}$ is a chain we can define processes $Z^1$ and $K^1$ on $[0,T]$ by putting $Z^1_t=Z^{1,k}_t$, $K^1_t=K^{1,k}_t$, $t\in[0,\tau_k]$. We see that $(Y^1,Z^1,K^1)$ is a solution of \underline{R}BSDE($\xi$,$\hat{f}+dV$,$L+Y^2$) on $[0,T]$. By \cite{EK}, $Y^2$ is a supermartingale, so by the Mertens decomposition, there exist $K^2\in\mathcal{V}^{1,+},Z^2\in\mathcal{H}$ such that
\[
Y^2_t=\int^T_t\,dK^2_r-\int^T_t Z^2_r\,dB_r,\quad t\in[0,T],
\]
and by \cite[Corollary 3.11]{KRzS},
\[
\int^T_0(Y^2_{r-}-\limsup_{s\uparrow r}(Y^1_s-U_s))\,dK^{2,*}_r
+\sum_{r<T}(Y^2_r-(Y^1_s-U_s))\Delta^+K^2_r=0.
\]
Therefore  $(Y^2,Z^2,K^2)$ is a solution of \underline{R}BSDE($0$,$0$,$Y^1-U$) on $[0,T]$. By \cite[Theorem 3.20]{KRzS} and Remark \ref{uwaga2.3}, $Z^1,Z^2\in\mathcal{H}^p$, $K^1,K^2\in\mathcal{S}^p$ if $p\ge 1$, and $Z^1,Z^2\in\mathcal{H}^q$, $q\in(0,1)$ and $K^1,K^2\in\mathcal{V}^{1,+}$ if $p=1$.

{\bf Step 4.} Write $Y:=Y^1-Y^2$, $Z:=Z^1-Z^2$, $R:=K^1-K^2$. We will show that $(Y,Z,R)$ is a solution of RBSDE($\xi$,$f+dV$,$L$,$U$). We have
\[
Y_t=\xi+\int^T_t f(r,Y_r)\,dr+\int^T_t\,dV_r+\int^T_t\,dR_r-\int^T_t Z_r\,dB_r,\quad t\in[0,T].
\]
Obviously  $L\le Y\le U$. The process $R$ satisfies the minimality condition because
\begin{align*}
\int^T_0 (&Y_{r-}-\limsup_{s\uparrow r} L_s)\,dR^{+,*}_r+\sum_{r<T}(Y_r-L_r)\Delta^+R^+_r\\ &\le \int^T_0 (Y^1_{r-}-\limsup_{s\uparrow r} (L_s+Y^2_s))\,dK^{1,*}_r
+\sum_{r<T}(Y^1_r-(L_r+Y^2_r))\Delta^+K^1_r=0
\end{align*}
and
\begin{align*}
\int^T_0 (&\liminf_{s\uparrow r} U_s-Y_{r-})\,dR^{-,*}_r+\sum_{r<T}(U_r-Y_r)\Delta^+R^-_r\\ &\le \int^T_0 (Y^2_{r-}-\limsup_{s\uparrow r}(Y^1_s-U_s))\,dK^{2,*}_r
+\sum_{r<T}(Y^2_r-(Y^1_r-U_r))\Delta^+K^2_r=0.
\end{align*}
The desired integrability of $Y$, $Z$ and $R$ follows from  Step 2 and Step 3. Furthermore,
\[
E\Big(\int^T_0|f(r,Y_r)|\,dr\Big)^p<\infty,\quad p\ge 1.
\]
Indeed, by Proposition \ref{proposition2.2} and Remark \ref{uwaga2.3},  $\underline{Y}\le Y\le \overline{Y}$, where $\underline{Y}$ is the first component of a solution of $\overline{\textnormal{R}}$BSDE($\xi$,$f+dV$,$U$) and $\overline{Y}$ is the first component of a solution of \underline{R}BSDE($\xi$,$f+dV$,$L$). By this and (H2),
$f(r,\overline{Y}_r)\le f(r,Y_r)\le f(r,\underline{Y}_r).$
Since by \cite[Theorem 4.1]{KRzS},
\[
E\Big(\int^T_0|f(r,\underline{Y}_r)|\,dr\Big)^p
+\Big(\int^T_0|f(r,\overline{Y}_r)|\,dr\Big)^p<\infty,\quad p\ge 1,
\]
the proof is complete.
\end{proof}

We close this subsection with  estimates for the difference of solutions of RBSDEs with  generators not  depending on $z$. We will use them in the next subsection to study the existence of solutions of general RBSDEs.

\begin{proposition}\label{prop2}
Let $(Y^i,Z^i,R^i)$ be  solutions of \textnormal{RBSDE}$(\xi,f^i+dV,L,U)$, $i=1,2$, where $f^1,f^2$ do not depend on $z$. Assume that $f^1$ satisfies \textnormal{(H2)}. If $(Y^1-Y^2)\in\mathcal{S}^p$ and $\int^T_0|f^1(r,Y^2_r)-f^2(r,Y^2_r)|\,dr\in\mathbb{L}^p$ for some $p>1$, then $(Z^1-Z^2)\in\mathcal{H}^p$ and there exists a constant $C_p$ depending only on $p$ such  that
\begin{align*}
&E\Big\{\sup_{t\le T}|Y^1_t-Y^2_t|^p+\Big(\int^T_0 |Z^1_r-Z^2_r|\,dr\Big)^{p/2}\Big\}
\le C_pE\Big(\int^T_0 |f^1(r,Y^2_r)-f^2(r,Y^2_r)|\,dr\Big)^p.
\end{align*}
\end{proposition}
%\end{document}
\begin{proof} Without loss of generality we may assume that $\mu\le 0$.
We know that
\begin{align*}
Y^1_t-Y^2_t&=\int^T_t \hat{f}(r,Y^1_r-Y^2_r)\,dr+\int^T_t\,d(R^1_r-R^2_r)-\int^T_t(Z^1_r-Z^2_r)\,dB_r,\quad t\in[0,T],
\end{align*}
where $\hat{f}(r,y)=f^1(r,y+Y^2_r)-f^2(r,Y^2_r)$.
Define $\tau_k=\inf\left\{t\in[0,T]:\int^t_0 |Z^1_r-Z^2_r|^2\,dr\ge k\right\}\wedge T$, $k\in\N$. By \cite[Corollary A.5]{KRzS} %\end{proof}
\begin{align}\label{lem1.0.0}
|Y^1_t-Y^2_t|^2&+\int^{\tau_k}_t|Z^1_r-Z^2_r|^2\,dr\le|Y^1_{\tau_k}-Y^2_{\tau_k}|^2
+2\int^{\tau_k}_t(Y^1_r-Y^2_r)\hat{f}(r,Y^1_r-Y^2_r)\,dr\nonumber
\\\nonumber
&\quad+2\int^{\tau_k}_t(Y^1_{r-}-Y^2_{r-})\,d(R^1_r-R^2_r)^*
+2\sum_{t\le r<\tau_k}(Y^1_r-Y^2_r)\Delta^+(R^1_r-R^2_r)
\\&\quad
-2\int^{\tau_k}_t(Y^1_r-Y^2_r)(Z^1_r-Z^2_r)\,dB_r
\end{align}
By the  minimality condition and the fact that  $U\ge Y^1\ge L$ and $U\ge Y^2\ge L$,
\begin{align}\label{lem1.1}\nonumber
&\int^{\tau_k}_t(Y^1_{r-}-Y^2_{r-})\,d(R^1_r-R^2_r)^*+\sum_{t\le r<\tau_k}(Y^1_r-Y^2_r)\Delta^+(R^1_r-R^2_r)\\\nonumber
&\quad\le\int^{\tau_k}_t(Y^1_{r-}-\limsup_{s\uparrow r}L_s)\,dR^{1,*,+}_r+\int^{\tau_k}_t(Y^2_{r-}-\limsup_{s\uparrow r}L_s)\,dR^{2,*,+}_r\\
&\nonumber\quad\quad+\sum_{t\le r<\tau_k}(Y^1_r-L_r)\Delta^+R^{1,+}_r+\sum_{t\le r<\tau_k}(Y^2_r-L_r)\Delta^+R^{2,+}_r
\\\nonumber
&\quad\quad+\int^{\tau_k}_t(\liminf_{s\uparrow r}U_s-Y^1_{r-})\,dR^{1,*,-}_r+\int^{\tau_k}_t(\liminf_{s\uparrow r}U_s-Y^2_{r-})\,dR^{2,*,-}_r\\
&\quad\quad+\sum_{t\le r<\tau_k}(U_r-Y^1_r)\Delta^+R^{1,-}_r+\sum_{t\le r<\tau_k}(U_r-Y^2_r)\Delta^+R^{2,-}_r
\le 0.
\end{align}
Since $\mu\le0$, using (H2) we get
\begin{equation}\label{lem1.2}
2y\hat{f}(r,y)\le2\mu|y|^2+2|y||\hat{f}(r,0)|\le 2|y||\hat{f}(r,0)|=2|y||f^1(r,Y^2_r)-f^2(r,Y^2_r)|.
\end{equation}
By (\ref{lem1.0.0})--(\ref{lem1.2}),
\begin{align*}
\int^{\tau_k}_0|Z^1_r-Z^2_r|^2\,dr&\le\sup_{t\le T}|Y^1_t-Y^2_t|^2+2\int^T_0|Y^1_r-Y^2_r||f^1(r,Y^2_r)-f^2(r,Y^2_r)|\,dr\\
&\quad+2\Big|\int^{\tau_k}_0(Y^1_r-Y^2_r)(Z^1_r-Z^2_r)\,dB_r\Big|.
\end{align*}
Since
\begin{align*}
&2\int^T_0|Y^1_r-Y^2_r||f^1(r,Y^2_r)-f^2(r,Y^2_r)|\,dr\\
&\qquad\le 2\sup_{t\le T}|Y^1_t-Y^2_t|\int^T_0|f^1(r,Y^2_r)-f^2(r,Y^2_r)|\,dr\\
&\qquad\le\sup_{t\le T}|Y^1_t-Y^2_t|^2+\Big(\int^T_0|f^1(r,Y^2_r)-f^2(r,Y^2_r)|\,dr\Big)^2,
\end{align*}
it follows from the above that
\begin{align}\label{lem1.3}
\nonumber
\Big(\int^{\tau_k}_0|Z^1_r-Z^2_r|^2\,dr\bigg)^{p/2}&\le C_p\Big(\sup_{t\le T}|Y^1_t-Y^2_t|^p+\bigg(\int^T_0|f^1(r,Y^2_r)-f^2(r,Y^2_r)|\,dr\Big)^p\\
&\quad\qquad\qquad+\Big|\int^{\tau_k}_0(Y^1_r-Y^2_r)(Z^1_r-Z^2_r)\,dB_r\Big|^{p/2}\Big).
\end{align}
By the Burkholder-Davis-Gundy inequality,
\begin{align*}
E&\Big|\int^{\tau_k}_0(Y^1_r-Y^2_r)(Z^1_r-Z^2_r)\,dB_r\Big|^{p/2}\le D_pE\Big(\int^{\tau_k}_0|Y^1_r-Y^2_r|^2|Z^1_r-Z^2_r|^2\,dr\Big)^{p/4}\\
&\qquad\qquad\qquad\le D_pE\Big[\sup_{t\le T}|Y^1_r-Y^2_r|^{p/2}\Big(\int^{\tau_k}_0|Z^1_r-Z^2_r|^2\,dr\Big)^{p/4}\Big].
\end{align*}
Hence
\[
E\Big|\int^{\tau_k}_0(Y^1_r-Y^2_r)(Z^1_r-Z^2_r)\,dB_r\Big|^{p/2}
\le \frac{D_p^2}{2}E\sup_{t\le T}|Y^1_r-Y^2_r|^p
+\frac{1}{2}E\Big(\int^{\tau_k}_0|Z^1_r-Z^2_r|^2\,dr\Big)^{p/2}.
\]
From the above and (\ref{lem1.3}), for every $k\ge 1$,
\[
E\Big(\int^{\tau_k}_0|Z^1_r-Z^2_r|^2\,dr\Big)^{p/2}\le C_pE\Big(\sup_{t\le T} |Y^1_t-Y^2_t|^p+\Big(\int^T_0|f^1(r,Y^2_r)-f^2(r,Y^2_r)|\,dr\Big)^{p}\Big)
\]
Letting $k\rightarrow\infty$ and applying  Fatou's lemma yields
\begin{align}
\label{eq3.0}
\nonumber
E\Big(\int^T_0 |Z^1_r-Z^2_r|^2\,dr\Big)^{p/2}
&\le C_p E\Big(\sup_{t\le T}|Y^1_t-Y^2_t|^p\\
&\quad+\Big(\int^T_0|f^1(r,Y^2_r)-f^2(r,Y^2_r)|\,dr\Big)^p\Big),
\end{align}
so  $(Z^1-Z^2)\in\mathcal{H}^p$.
To show the desired estimate for $Y^1-Y^2$, we use once again
\cite[Corollary A.5]{KRzS}. We have
\begin{align}
\label{prop2.2.0}
\nonumber
|Y^1_t-Y^2_t|&\le \int^T_t\mathrm{sgn}(Y^1_r-Y^2_r)\hat{f}(r,Y^1_r-Y^2_r)\,dr
+\int^T_t\mathrm{sgn}(Y^1_{r-}-Y^2_{r-})\,d(R^1_r-R^2_r)^*\\
&\quad+\sum_{t\le r< T}\mathrm{sgn}(Y^1_r-Y^2_r)\Delta^+(R^1_r-R^2_r)
-\int^T_t\mathrm{sgn}(Y^1_r-Y^2_r)(Z^1_r-Z^2_r)\,dB_r.
\end{align}
Since  $f^1$ satisfies (H2) and $\mu\le 0$,
\begin{equation}\label{prop2.2}
\mathrm{sgn}(y)\hat{f}(y)\le \,\lambda\mu|y|+|\hat{f}(r,0)|\le|\hat{f}(r,0)|=|f^1(r,Y^2_r)-f^2(r,Y^2_r)|.
\end{equation}
Since
$\mathrm{sgn}(Y^1_r-Y^2_r)=|Y^1_r-Y^2_r|^{-1}(Y^1_r-Y^2_r)\mathbf{1}_{\{Y^1_r-Y^2_r\neq 0\}}$,
it follows from  (\ref{lem1.1}), (\ref{prop2.2.0}) and (\ref{prop2.2}) that
\begin{align*}
|Y^1_t-Y^2_t|\le \int^T_0|f^1(r,Y^2_r)-f^2(r,Y^2_r)|\,dr-\int^T_t\mathrm{sgn}(Y^1_r-Y^2_r)(Z^1_r-Z^2_r)\,dB_r.
\end{align*}
Hence
\begin{align*}
|Y^1_t-Y^2_t|\le E\Big(\int^T_0|f^1(r,Y^2_r)-f^2(r,Y^2_r)|\,dr\,|\mathcal{F}_t\Big)
\end{align*}
since $\int^{\cdot}_0\mathrm{sgn}(Y^1_r-Y^2_r)(Z^1_r-Z^2_r)\,dB_r$ is a uniformly integrable martingale. This implies that
\begin{align*}
E\sup_{t\le T}|Y^1_t-Y^2_t|^p&\le E\Big(\sup_{t\le T} E\Big(\int^T_0|f^1(r,Y^2_r)-f^2(r,Y^2_r)|\,dr\,|\mathcal{F}_t\Big)^p\Big).
\\
&\le C_pE\Big(\int^T_0|f^1(r,Y^2_r)-f^2(r,Y^2_r)|\,dr\Big)^p.
\end{align*}
Combining the above inequality with (\ref{eq3.0}) completes the proof.
\end{proof}

\subsubsection{Existence of solutions for general $f$}

\begin{theorem}
Assume that $p=1$ and \textnormal{(H1)-(H5), (H6*), (Z)} are satisfied. Then there exists a unique solution $(Y,Z,R)$ of \textnormal{RBSDE}$(\xi,f+dV,L,U)$ such that $Y$ is of class \textnormal{(D)}, $Z\in\mathcal{H}^q$, $q\in(0,1)$ and $R\in\mathcal{V}^1$.
\end{theorem}
\begin{proof}
Without loss of generality we may assume that $\mu=0$. We will use Picard's iteration method. Set $(Y^0,Z^0,R^0)=(0,0,0)$ and for  $n\ge 0$ define
\[
Y^{n+1}_t=\xi+\int^T_t f(r,Y^{n+1}_r,Z^n_r)\,dr+\int^T_t\,dV_r+\int^T_t\,dR^{n+1}_r-\int^T_t Z^{n+1}_r\,dB_r,\quad t\in[0,T].
\]
Let $f_n(r,y)=f(r,y,Z^{n-1}_r)$. For each $n\ge 0$, $(Y^n,Z^n,R^n)$ is a solution of RBSDE$(\xi,f_n+dV,L,U)$ such that $Y^n$ is of class (D), $Z\in\mathcal{H}^q$, $q\in(0,1)$, and $R\in\mathcal{V}^1$. The existence of this solution follows from Theorem 2.7, since by induction, $E\int^T_0|f_n(r,0)|+|f_n(r,X_r)|\,dr<\infty$, $n\ge 1$.
By \cite[Corollary A.5]{KRzS} and the minimality conditions for $R^{n+1}$ and $R^n$, for any $\sigma,\tau\in\Gamma$, such that  $\sigma\le\tau$ we have
\begin{align}\label{th2.0}
\nonumber
|Y^{n+1}_{\sigma}-Y^n_{\sigma}|&\le \int^{\tau}_{\sigma}\sgn (Y^{n+1}_r-Y^n_r)(f(r,Y^{n+1}_r,Z^n_r)-f(r,Y^n_r,Z^{n-1}_r))\,dr\\
&\quad-\int^{\tau}_{\sigma}\sgn (Y^{n+1}_r-Y^n_r)(Z^{n+1}_r-Z^n_r)\,dB_r,\quad t\in[0,T].
\end{align}
By (H2) and (Z),
\begin{align*}
&\sgn (Y^{n+1}_r-Y^n_r)(f(r,Y^{n+1}_r,Z^n_r)-f(r,Y^n_r,Z^{n-1}_r))\le\sgn (Y^{n+1}_r-Y^n_r)(f(r,Y^n_r,Z^n_r)\\
&\quad-f(r,Y^n,0)+f(r,Y^n,0)-f(r,Y^n_r,Z^{n-1}_r))\le 2\gamma(g_r+|Y^n_r|+|Z^n_r|+|Z^{n-1}_r|)^{\alpha}.
\end{align*}
Let $\{\tau_k\}$ be a localizing sequence for the local martingale $\int^{\cdot}_{\sigma}\sgn (Y^{n+1}_r-Y^n_r)(Z^{n+1}_r-Z^n_r)\,dB_r$. By the above inequality and (\ref{th2.0}), for $n\ge 1$ we have
\[
|Y^{n+1}_{\sigma}-Y^n_{\sigma}|\le E\Big(|Y^{n+1}_{\tau_k}-Y^n_{\tau_k}|+2\gamma\int^T_0 (g_r+|Y^n_r|+|Z^n_r|+|Z^{n-1}_r|)^{\alpha}\,dr|\mathcal{F}_{\sigma}\Big).
\]
Letting $k\rightarrow\infty$ and using the fact that $Y^{n+1},Y^n$ are of class (D) we get
\begin{equation}\label{th2.0.0}
|Y^{n+1}_{\sigma}-Y^n_{\sigma}|\le 2\gamma E\Big(\int^T_0 (g_r+|Y^n_r|+|Z^n_r|+|Z^{n-1}_r|)^{\alpha}\,dr|\mathcal{F}_{\sigma}\Big).
\end{equation}
Note that $Z^n,Z^{n-1}\in\mathcal{H}^q$, $q\in(0,1)$, $Y^n$ is of class (D) and $\{g_t\}_{t\in[0,T]}$ is integrable. Therefore the random variable
$I_n:=\int^T_0 (g_r+|Y^n_r|+|Z^n_r|+|Z^{n-1}_r|)^{\alpha}\,dr$
belongs to $\mathbb{L}^q$ supposing that $\alpha\cdot q<1$. Fix $\bar{q}\in(1,2)$ such that $\alpha\cdot \bar{q}<1$. Then, by  Doob's inequality and (\ref{th2.0.0}), $Y^{n+1}-Y^n\in\mathcal{S}^{\bar{q}}$ for $n\ge 1$. Note that
\[
Y^{n+1}_t-Y^n_t=\int^T_t \hat f_n(r,Y^{n+1}_r-Y^n_r)\,dr+\int^T_t\,d(R^{n+1}_r-R^n_r)-\int^T_t(Z^{n+1}_r-Z^n_r)\,dB_r,\quad t\in[0,T],
\]
where $\hat f_n(r,y)=f(r,y+Y^n_r,Z^n_r)-f(r,Y^n_r,Z^{n-1}_r)$. By (Z) we have
\begin{align*}
\int^T_0|f(r,Y^n_r,Z^n_r)-f(r,Y^n_r,Z^{n-1}_r)|\,dr\le C\int^T_0 (g_r+|Y^n_r|+|Z^n_r|+|Z^{n-1}_r|)^{\alpha}.
\end{align*}
Since $I_n\in\mathbb{L}^{\bar{q}}$, it follows from  Proposition \ref{prop2} that
$Z^{n+1}-Z^n\in\mathcal{H}^{\bar{q}}$  and  there exists a constant $C_{\bar{q}}$ such that for all $n\ge 1$,
\[
\|(Y^{n+1}-Y^n,Z^{n+1}-Z^n)\|^{\bar{q}}\le C_{\bar{q}}E\Big(\Big(\int^T_0|f(r,Y^n_r,Z^n_r)-f(r,Y^n_r,Z^{n-1}_r)|\,dr\Big)^{\bar{q}}\Big),
\]
where
\[
\|(Y,Z)\|=\Big(E\Big(\sup_{t\le T}|Y_t|^{\bar{q}}+\Big(\int^T_0|Z_r|^2\,dr\Big)^{\bar{q}/{2}}
\Big)\Big)^{1/\bar{q}}.
\]
Since $f$ satisfies (H1), using H\"older's inequality we get
\[
\|(Y^{n+1}-Y^n,Z^{n+1}-Z^n)\|^{\bar{q}}\le CE\Big(\Big(\int^T_0|Z^n_r-Z^{n-1}_r|^2\,dr\Big)^{\bar{q}/2}\Big)
\]
for $n\ge2$, where $C=C_{\bar{q}}\lambda^{\bar{q}} T^{\bar{q}/2}$. Therefore, for $n\ge 2$,
\[
\|(Y^{n+1}-Y^n,Z^{n+1}-Z^n)\|^{\bar{q}}\le C^{n-1}\|(Y^2-Y^1,Z^2-Z^1)\|^{\bar{q}}.
\]
If $C=C_{\bar{q}}\lambda^{\bar{q}} T^{\bar{q}/2}<1$, then using the above inequality one can deduce that  $\{(Y^n-Y^1,Z^n-Z^1)\}$ is a Cauchy sequence, so $(Y^n-Y^1,Z^n-Z^1)$ converges to some $(U,V)$ in $\mathcal{S}^{\bar{q}}\times\mathcal{H}^{\bar{q}}$. Since $(Y^1,Z^1)\in\mathcal{S}^{\beta}\times\mathcal{H}^{\beta}$, $\beta\in(0,1)$, it follows that  $(Y^n,Z^n)$ converges to $(Y,Z):=(U+Y^1,V+Z^1)$ in $\mathcal{S}^{\beta}\times\mathcal{H}^{\beta}$, $\beta\in(0,1)$. Moreover, $Y^1$ is of class (D), so $Y^n\rightarrow Y$ in the norm $\|\cdot\|_1$. In the general case, we divide $[0,T]$ into a finite number of small intervals and use the standard argument.

Let $\bar{f}(r,y)=f(r,y,Z_r)$ and  $(\bar{Y},\bar{Z},\bar{R})$ be a solution of RBSDE($\xi$,$\bar{f}+dV$,$L$,$U$) such that $\bar{Y}$ is of class (D), $\bar{Z}\in\mathcal{H}^q$, $q\in(0,1)$, $\bar{R}\in\mathcal{V}^1$. The existence of the solution follows from Theorem \ref{thmz}. Repeating the reasoning following (\ref{th2.0}), but with $Y^{n+1}$ replaced by $\bar{Y}$, we get
\[
\|(\bar{Y}-Y^n,\bar{Z}-Z^n)\|^{\bar{q}}\le C_{\bar{q}}E\Big(\Big(\int^T_0|f(r,Y^n_r,Z_r)-f(r,Y^n_r,Z^{n-1}_r)|\,dr\Big)^{\bar{q}}\Big).
\]
Since $f$ satisfies (H1), using H\"older's inequality we obtain
\[
\|(\bar{Y}-Y^n,\bar{Z}-Z^n)\|^{\bar{q}}\le CE\Big(\Big(\int^T_0|Z_r-Z^{n-1}_r|^2\,dr\Big)^{\bar{q}/2}\Big)\rightarrow 0,\,n\rightarrow\infty
\]
for $n\ge2$, where $C=C_{\bar{q}}\lambda^{\bar{q}} T^{\bar{q}/2}$. Hence $(\bar{Y},\bar{Z})=(Y,Z)$. Therefore the triple $(Y,Z,\bar{R})$ is a solution of RBSDE($\xi$,$f+dV$,$L$,$U$). This completes the proof.
\end{proof}

\begin{theorem}
Assume that $p>1$ and \mbox{\rm(H1)-(H6)} are satisfied. Then there exists a unique solution $(Y,Z,R)$ of \textnormal{RBSDE}$(\xi,f+dV,L,U)$ such that $Y\in\mathcal{S}^p$, $Z\in\mathcal{H}^p$ and $R\in\mathcal{V}^p$.
\end{theorem}
\begin{proof}
Consider the space $\mathcal{S}^p\oplus\mathcal{H}^p$ equipped in the norm
\[
\|(Y,Z)\|_{\mathcal{S}^p\oplus\mathcal{H}^p}=\Big(E\Big(\sup_{t\le T} |Y_t|^p+\Big(\int^T_0|Z_r|^2\,dr\Big)^{{p}/{2}}\Big)\Big)^{{1}/{p}}.
\]
Define $\phi:\mathcal{S}^p\oplus\mathcal{H}^p\longrightarrow\mathcal{S}^p\oplus\mathcal{H}^p$
as $\phi((G,H))=(Y,Z)$, where $(Y,Z,R)$ is the unique solution of RBSDE($\xi$,$\hat{f}+dV$,$L$,$U$) with $\hat{f}(t,y)=f(t,y,H_t)$ such that $Y\in\mathcal{S}^p$, $Z\in\mathcal{H}^p$, $R\in\mathcal{V}^p$. The existence and uniqueness of such solution follows from Theorem \ref{thmz}. Let $(Y^1,Z^1),(Y^2,Z^2)\in\mathcal{S}^p\oplus\mathcal{H}^p$ and  $(G,H),(G',H')\in\mathcal{S}^p\oplus\mathcal{H}^p$ be such that $(Y,Z)=\phi(G,H)$ and $(Y',Z')=\phi(G',H')$. By Proposition \ref{prop2}, there exists a constant $C_p$ such that
\[
\|(Y-Y',Z-Z')\|_{\mathcal{S}^p\oplus\mathcal{H}^p}\le C_pE\Big(\Big(\int^T_0|f(r,Y'_r,H_r)-f(r,Y'_r,H'_r)|\,dr\Big)^{p}\Big).
\]
Since $f$ satisfies (H1), applying H\"older's inequality yields
\[
\|(Y-Y',Z-Z')\|_{\mathcal{S}^p\oplus\mathcal{H}^p}\le CE\Big(\Big(\int^T_0|H_r-H'_r|^2\,dr\Big)^{{p}/{2}}\Big)
\]
with $C=C_p\lambda^p T^{{p}/{2}}$. Hence
\[
\|(Y-Y',Z-Z')\|_{\mathcal{S}^p\oplus\mathcal{H}^p}\le C\|(G-G',H-H')\|_{\mathcal{S}^p\oplus\mathcal{H}^p}.
\]
If $C=C_p\lambda^p T^{{p}/{2}}<1$, then  $\phi$ is a contraction, so by Banach's fixed-point theorem,  there exists $(\bar{Y},\bar{Z})$ such that $\phi(\bar{Y},\bar{Z})=(\bar{Y},\bar{Z})$. We set $\bar{R}=R$. Then the triple $(\bar{Y},\bar{Z},\bar{R})$ is a solution of RBSDE$(\xi,f+dV,L,U)$. To get the existence of a solution in the general case we divide $[0,T]$ into a finite number of small intervals and use the standard argument.
\end{proof}

\nsubsection{Penalization methods for RBSDEs with two regulated barriers}
\label{sec4}

In this section we assume additionally that the barriers $L,U$ are $\mathbb{F}$-adapted regulated processes. We  consider approximation of the solution of RBSDE($\xi$,$f+dV$,$L$,$U$) by modified penalization methods.

\subsubsection{Monotone penalization method via RBSDEs}

By \cite[Theorem 3.20]{KRzS}, for each $n\ge1$ there exists a solution   $(\bar{Y}^n,\bar{Z}^n,\bar{A}^n)$ of $\mathrm{\overline{R}BSDE}$, with upper barrier $U$, of the form
\begin{align}\nonumber
\bar{Y}^n_t&=\xi+\int^T_t f(r,\bar{Y}^n_r,\bar{Z}^n_r)\,dr+\int^T_t\,dV_r-\int^T_t\,d\bar{A}^n_r-\int^T_t \bar{Z}^n_r\,dB_r\\
&\quad+n\int^T_t(\bar{Y}^n_r-L_r)^-\,dr
+\sum_{t\le\sigma_{n,i}<T}(\bar{Y}^n_{\sigma_{n,i}+}+\Delta^+V_{\sigma_{n,i}}-L_{\sigma_{n,i}})^-\label{eq4.1}
\end{align}
such that if $p>1$, then $\bar{Y}^n\in\mathcal{S}^p$, $\bar{Z}^n\in\mathcal{H}^p$, $\bar{A}^n\in\mathcal{V}^{+,p}$, and if $p=1$, then $\bar{Y}^n$ is of class (D), $\bar{Z}^n\in\mathcal{H}^q$, $q\in(0,1)$, $\bar{A}^n\in\mathcal{V}^{+,1}$. In (\ref{eq4.1}),  $\{\{\sigma_{n,i}\}\}$ is an array of stopping times exhausting  the right-side jumps of $L$ and $V$. It is defined inductively as follows. We set $\sigma_{1,0}=0$, and then
\[
\sigma_{1,i}=\inf\{t>\sigma_{1,i-1};\,\Delta^+L_t<-1\,\,\mathrm{or}\,\,\Delta^+V_t<-1\}\wedge T,\,i=1,\dots,k_1
\]
for some $k_1\in\N$. Next, for $n\in\N$ and given array $\{\{\sigma_{n,i}\}\}$, we set $\tilde{\sigma}_{n+1,0}=0$,
\[
\tilde{\sigma}_{n+1,i}=\inf\{t>\tilde{\sigma}_{n+1,i-1};\,\Delta^+L_t<-1/(n+1)\,\,\mathrm{or}\,\,\Delta^+V_t<-1/(n+1)\}\wedge T,\,i\ge 1.
\]
Let $j_{n+1}$ be chosen so that $P(\tilde{\sigma}_{n+1,j_{n+1}}<T)\le\frac1n$. We put
\[\sigma_{n+1,i}=\tilde{\sigma}_{n+1,i},\quad i=1,\dots,j_{n+1},\quad\sigma_{n+1,i+j_{n+1}}
=\tilde{\sigma}_{n+1,j_{n+1}}\vee \sigma_{n,i},\quad i=1,....,k_n,\]
$k_{n+1}=j_{n+1}+k_{n}+1$. Finally, we put $\sigma_{n+1,k_{n+1}}=T$. Since $\Delta^+L_r<-1/n$ or $\Delta^+V_t<-1/n$ implies that $\Delta^+L_r<-1/(n+1)$ or $\Delta^+V_t<-1/(n+1)$, it follows from our construction that
\begin{equation}
\label{eq4.2}
\bigcup^{k_n}_{i=1} [[\sigma_{n,i}]]\subset \bigcup^{k_{n+1}}_{i=1} [[\sigma_{n+1,i}]]\quad n\in\N.\end{equation}
Observe that, on  each interval $(\sigma_{n,i-1},\sigma_{n,i}]$,
$i=1,\dots,{k_n+1}$, the triple
$(\bar{Y}^n,\bar{Z}^n,\bar{A}^n)$ is a solution of the classical $\mathrm{\overline{R}BSDE}$ of the form
\begin{align*}
\bar{Y}^n_t&=L_{\sigma_{n,i}}\vee(\bar{Y}^n_{\sigma_{n,i}+}+\Delta^+V_{\sigma_{n,i}})\wedge U_{\sigma_{n,i}}+\int^{\sigma_{n,i}}_t f(r,\bar{Y}^n_r,\bar{Z}^n_r)\,dr+\int^{\sigma_{n,i}}_t\,dV_r\\
&\quad-\int^{\sigma_{n,i}}_t\,d\bar{A}^n_r-\int^{\sigma_{n,i}}_t \bar{Z}^n_r\,dB_r+n\int^{\sigma_{n,i}}_t(\bar{Y}^n_r-L_r)^-\,dr,\quad t\in(\sigma_{n,i-1},\sigma_{n,i}]
\end{align*}
with upper barrier $U$ and $\bar{Y}^n_0=L_0\vee(\bar{Y}^n_{0+}+\Delta^+V_0)\wedge U_0$, $n\in\mathbb{N}$. Therefore, to solve equation (\ref{eq4.1}), we  divide $[0,T]$ into a finite number of intervals $[0,\sigma_{n,1}],\dots(\sigma_{n,k_n-1},T]$ and we solve the equation on each interval $(\sigma_{n,i-1},\sigma_{n,i}]$ starting from $(\sigma_{n,k_n-1},T]$.

Note that  (\ref{eq4.1}) can be written in the shorter form
\begin{equation}
\label{eq4.3}
\bar{Y}^n_t=\xi+\int_t^Tf(r,\bar{Y}^n_r,\bar{Z}^n_r)\,dr+\int^T_t\,dV_r
+\int^T_t\,d\bar{A}^n_r-\int_t^T\bar{Z}^n_r\,dB_r
+\int^T_t\,dK^n_r,
\end{equation}
where
\begin{align}
\label{eq4.4} K^n_t&=n\int_0^t (\bar{Y}^n_r-L_r)^-\,dr+\sum_{0\leq \sigma_{n,i}<t}(\bar{Y}^n_{\sigma_{n,i}+}+\Delta^+V_{\sigma_{n,i}}-L_{\sigma_{n,i}})^-\nonumber \\
&\equiv K^{n,*}_t+K^{n,d}_t,\quad t\in[0,T].
\end{align}

\begin{proposition}\label{stw.4.1}
Let $\{\sigma_k;k=0,\dots,m\}$ be an increasing sequence of stopping times such that $\sigma_0=0$ and $\sigma_m=T$. Let $(Y^i,Z^i,A^i)$, $i=1,2$, be a solution of the  $\mathrm{\overline{R}BSDE}$
\begin{align}
\nonumber
Y^i_t&=\xi^i+\int^T_t f^i(s,Y^i_r,Z^i_r)\,dr+\int^T_t\,dV_r-\int^T_t\,dA^i_r-\int^T_t Z^i_r\,dB_r\\
&\quad+\sum_{t\le\sigma_k<T}(Y^i_{\sigma_k+}+\Delta^+V_{\sigma_k}-L_{\sigma_k})^-,\quad t\in[0,T],\label{eq4}
\end{align}
with upper barrier $U$, such that if $p>1$, then  $Y^i\in\mathcal{S}^p$, $Z^i\in\mathcal{H}^p$, $A^i\in\mathcal{V}^{+,p}$, and if $p=1$, then $Y^i$ is of class \textnormal{(D)}, $Z^i\in\mathcal{H}^q$, $q\in(0,1)$, $A^i\in\mathcal{V}^{+,1}$, $i=1,2$. Assume that $p>1$ and \textnormal{(H1)--(H6)} are satisfied or $p=1$ and \textnormal{(H1)--(H5), (H6*), (Z)} are satisfied. Let $\xi^1\le\xi^2$, $f^1\le f^2$. Then $dA^1\le dA^2$.
\end{proposition}
\begin{proof}
For $i=1,2$ we have
\begin{align*}
Y^i_t&=L_{\sigma_{k+1}}\vee(Y^i_{\sigma_{k+1}+}+\Delta^+V^i_{\sigma_{k+1}})\wedge U_{\sigma_{k+1}}+\int^{\sigma_{k+1}}_t f^i(s,Y^i_r,Z^i_r)\,dr+\int^{\sigma_{k+1}}_t\,dV_r\\
&\quad-\int^{\sigma_{k+1}}_t\,dA^i_r-\int^{\sigma_{k+1}}_t Z^i_r\,dB_r,\quad t\in({\sigma_k},{\sigma_{k+1}}],\quad k=0,\dots,m-1.
\end{align*}
By Proposition \ref{proposition2.2} and Lemma \ref{uw2.0}, $Y^1\le Y^2$, so $L_{\sigma_{k+1}}\vee(Y^1_{\sigma_{k+1}+}+\Delta^+V^1_{\sigma_{k+1}})\wedge U_{\sigma_{k+1}}\le L_{\sigma_{k+1}}\vee(Y^2_{\sigma_{k+1}+}+\Delta^+V^2_{\sigma_{k+1}})\wedge U_{\sigma_{k+1}}$. Now on  all intervals $({\sigma_k},{\sigma_{k+1}}]$ we consider introduced in  \cite{KRzS} penalization schemes for the $\mathrm{\overline{R}}$BSDE($L_{\sigma_{k+1}}\vee(Y^i_{\sigma_{k+1}+}+\Delta^+V^i_{\sigma_{k+1}})\wedge U_{\sigma_{k+1}}$,$f^i+dV$,$U$), $i=1,2$. They have  forms
\begin{align*}
Y^{i,n}_t&=L_{\sigma_{k+1}}\vee(Y^{i,n}_{\sigma_{k+1}+}+\Delta^+V_{\sigma_{k+1}})\wedge U_{\sigma_{k+1}}+\int^{\sigma_{k+1}}_t f^i(s,Y^{i,n}_r,Z^{i,n}_r)\,dr+\int^{\sigma_{k+1}}_t\,dV_r\\
&\quad-n\int^{\sigma_{k+1}}_t(Y^{i,n}_r-U_r)^+\,dr+\sum_{t\le\tau^k_{n,j}<\sigma_{k+1}}(Y^{i,n}_{\tau^k_{n,j}}+\Delta^+V_{\tau^k_{n,j}}-U_{\tau^k_{n,j}})^+\\
&\quad-\int^{\sigma_{k+1}}_t Z^{i,n}_r\,dB_r,\quad t\in({\sigma_k},{\sigma_{k+1}}],\quad k=0,\dots,m-1,\,n\in\mathbb{N},
\end{align*}
where $\{\{\tau^k_{n,j}\}\}$, is an array of stopping times exhausting  the right-side jumps of $U$ and $V$ defined similarly to the array $\{\sigma_{n,i}\}\}$ for  (\ref{eq4.1}).
Set
\begin{equation}\label{eq4.0}
\int^{\sigma_{k+1}}_t\,dA^{i,n}_r:=n\int^{\sigma_{k+1}}_t(Y^{i,n}_r-U_r)^+\,dr-\sum_{t\le\tau^k_{n,j}<\sigma_{k+1}}(Y^{i,n}_{\tau^k_{n,j}}+\Delta^+V_{\tau^k_{n,j}}-U_{\tau^k_{n,j}})^+.
\end{equation}
By Proposition \ref{proposition2.2} and Lemma \ref{uw2.0}, $Y^{1,n}\le Y^{2,n}$. By this and (\ref{eq4.0}), $dA^{1,n}\le dA^{2,n}$. Furthermore, by  \cite[Lemma 4.1]{kl} and  \cite[Theorem 4.1]{KRzS}, $A^{1,n}_{\tau}\rightarrow A^1_{\tau}$, $A^{2,n}_{\tau}\rightarrow A^2_{\tau}$ weakly in $L^1$ for every $\tau\in\Gamma$. Therefore, by the Section Theorem, $dA^1\le dA^2$ on $({\sigma_k},{\sigma_{k+1}}]$. In order to complete the proof we have to show that $\Delta^+A^1_{\sigma_k}\le\Delta^+A^2_{\sigma_k}$, $k=i,\dots,m-1$. If $\Delta^+A^1_{\sigma_k}=0$, then this inequality is obvious. Let $\Delta^+A^1_{\sigma_k}>0$. Note that
\begin{equation}\label{eq4.0.1}
\Delta^+A^1_{\sigma_k}=\Delta^+Y^1_{\sigma_k}
+\Delta^+V_{\sigma_k}+(Y^1_{\sigma_k+}+\Delta^+V_{\sigma_k}-L_{\sigma_k})^-.
\end{equation}
By the the minimality condition, $Y^1_{\sigma_k}=U_{\sigma_k}$. Therefore
\begin{equation}\label{eq4.0.2}
\begin{split}
\Delta^+Y^1_{\sigma_k}=Y^1_{\sigma_k+}-Y^1_{\sigma_k}\le Y^2_{\sigma_k+}-Y^1_{\sigma_k}=Y^2_{\sigma_k+}-U_{\sigma_k}\le Y^2_{\sigma_k+}-Y^2_{\sigma_k}=\Delta^+Y^2_{\sigma_k}.
\end{split}
\end{equation}
If $(Y^1_{\sigma_k+}+\Delta^+V_{\sigma_k}-L_{\sigma_k})^-=0$, then by (\ref{eq4.0.1}) and (\ref{eq4.0.2}),
\begin{align*}
\Delta^+A^1_{\sigma_k}\le \Delta^+Y^2_{\sigma_k}+\Delta^+V_{\sigma_k}&\le \Delta^+Y^2_{\sigma_k}+\Delta^+V_{\sigma_k}
+(Y^2_{\sigma_k+}+\Delta^+V_{\sigma_k}-L_{\sigma_k})^-\\
&=\Delta^+A^2_{\sigma_k}.
\end{align*}
If $(Y^1_{\sigma_k+}+\Delta^+V_{\sigma_k}-L_{\sigma_k})^-\neq 0$, then by (\ref{eq4.0.1}) we have
\[
\Delta^+A^1_{\sigma_k}=-Y^1_{\sigma_k}+L_{\sigma_k}=-U_{\sigma_k}+L_{\sigma_k}\le 0,
\]
which is a contradiction. Hence $\Delta^+A^1_{\sigma_k}\le\Delta^+A^2_{\sigma_k}$, which completes the proof.
\end{proof}

\begin{theorem}\label{r.6}
Let $(\bar{Y}^n,\bar{Z}^n,\bar{A}^n)$, $n\in\mathbb{N}$, be defined by \mbox{\rm(\ref{eq4.1})}.
\begin{enumerate}[\rm(i)]
\item Assume that $p>1$ and \textnormal{(H1)--(H6)} are satisfied. Then $\bar{Y}^n_t\nearrow Y_t$, $t\in[0,T]$, and for every $\gamma\in[1,2)$,
\begin{equation}\label{Tw4.1.0.0}
E\Big(\int^T_0|\bar{Z}^n_r-Z_r|^{\gamma}\,dr\Big)^{p/2}\rightarrow 0,
\end{equation}
where $(Y,Z,R)$ is the unique solution of \textnormal{RBSDE}($\xi$,$f+dV$,$L$,$U$) such that $Y\in\mathcal{S}^p$, $Z\in\mathcal{H}^p$, $R\in\mathcal{S}^p$. Moreover, if $\Delta^-R^+_t=0$ for $t\in(0,T]$, then the above convergence also holds with $\gamma=2$, and moreover, $|\bar Y^n-Y|_p\rightarrow 0$.

\item Assume that  $p=1$ and  \textnormal{(H1)--(H5), (H6*)} and \textnormal{(Z)} are satisfied. Then $\bar{Y}^n_t\nearrow Y_t$, $t\in[0,T]$,  and for all $\gamma\in[1,2)$ and $r\in(0,1)$,
\begin{equation}\label{Tw4.1.0.1}
E\Big(\int^T_0|\bar{Z}^n_r-Z_r|^{\gamma}\,dr\Big)^{{r}/{2}}\rightarrow 0,
\end{equation}
where $(Y,Z,R)$ is the unique solution of \textnormal{RBSDE}($\xi$,$f+dV$,$L$,$U$) such that $Y$ is of class \textnormal{(D)}, $Z\in\mathcal{H}^q$, $R\in\mathcal{V}^1$, $q\in(0,1)$.
Moreover, if $\Delta^-R^+_t=0$ for $t\in(0,T]$, then the above convergence also hold with $\gamma=2$
and $|\bar Y^n-Y|_1\rightarrow 0$.
\end{enumerate}
\end{theorem}
\begin{proof}
Without loss of generality we may assume that $\mu=0$.

{\bf Step 1.} We show that for every $n\in\mathbb{N}$ the triple $(\bar{Y}^n,\bar{Z}^n,K^n-\bar{A}^n)$ is a solution of RBSDE($\xi$,$f+dV$,$L^n$,$U$) with $L^n=L-(\bar{Y}^n-L)^-=L\wedge \bar{Y}^n$. It is clear that $\bar{Y}^n_t\ge L^n_t$, $t\in[0,T]$. We also have
\[
\int_0^T (\bar{Y}^n_{r-}-L^n_{r-})\,dK^{n,*}_r=n\int^T_0 (\bar{Y}^n_r-L^n_r)(\bar{Y}^n_r-L_r)^-\,dr=n\int^T_0 (\bar{Y}^n_r-L_r)^+(\bar{Y}^n_r-L_r)^-\,dr=0
\]
and
\begin{align*}
\sum_{r<T}(\bar{Y}^n_r-L^n_r)\Delta^+K^n_r&=\sum_{\sigma_{n,i}<T}(\bar{Y}^n_{\sigma_{n,i}}-L^n_{\sigma_{n,i}})(\bar{Y}^n_{\sigma_{n,i}+}+\Delta^+V_{\sigma_{n,i}}-L_{\sigma_{n,i}})^-\\
&=\sum_{\sigma_{n,i}<T}(\bar{Y}^n_{\sigma_{n,i}}-L_{\sigma_{n,i}})^+(\bar{Y}^n_{\sigma_{n,i}+}+\Delta^+V_{\sigma_{n,i}}-L_{\sigma_{n,i}})^-=0.
\end{align*}
We will justify the last equality. Striving for  contradiction, suppose that
\begin{align*}
\sum_{\sigma_{n,i}<T}(\bar{Y}^n_{\sigma_{n,i}}
-L_{\sigma_{n,i}})^+(\bar{Y}^n_{\sigma_{n,i}+}+\Delta^+V_{\sigma_{n,i}}
-L_{\sigma_{n,i}})^-\neq 0.
\end{align*}
Then there exists $i\in\{1,\ldots,k_n\}$ such that $\bar{Y}^n_{\sigma_{n,i}}-L_{\sigma_{n,i}}>0$ and $\bar{Y}^n_{\sigma_{n,i}+}+\Delta^+V_{\sigma_{n,i}}-L_{\sigma_{n,i}}<0$. By the last inequality and (\ref{eq4.3}), $\Delta^+\bar{Y}^n_{\sigma_{n,i}}=\Delta^+K^n_{\sigma_{n,i}}-\Delta^+V_{\sigma_{n,i}}=(\bar{Y}^n_{\sigma_{n,i}+}+\Delta^+V_{\sigma_{n,i}}-L_{\sigma_{n,i}})-\Delta^+V_{\sigma_{n,i}}$. Hence $\bar{Y}^n_{\sigma_{n,i}}=L_{\sigma_{n,i}}$, which is a contradiction.

{\bf Step 2.} We will show that there exists a process $Z'\in \HH$ and a chain $\{\tau_k\}$ such that
\begin{equation}\label{eq.kled.ert}
E\int^{\tau_k}_0|\bar{Z}^n_s-Z'_s|^{\gamma}\,ds\rightarrow 0,\quad\gamma\in[1,2).
\end{equation}
Moreover, we will show that if $p>1$, then   $Z'\in \mathcal{H}^p$  and (\ref{Tw4.1.0.0}) holds with $Z$ replaced by $Z'$, and if  $p=1$, then  $Z'\in \mathcal{H}^q$, $q\in (0,1)$, and
(\ref{Tw4.1.0.1}) holds with $Z$ replaced by $Z'$.
To show this we will use \cite[Lemma 4.2]{kl}. If $p>1$, then by (H6) there exists $X\in\mathcal{M}_{loc}+\mathcal{V}^p$, $X\in\mathcal{S}^p$ such that $X\ge L$ and $\int^T_0 f^-(s,X_s,0)\,ds\in \mathbb L^p$. If $p=1$, then by (H6*) there exists $X\in\mathcal{M}_{loc}+\mathcal{V}^1$ of class (D) such that  $X\ge L$ and $\int^T_0 f^-(s,X_s,0)\,ds\in \mathbb L^1$. Since the Brownian fitration has the representation property, there exist processes $H\in\mathcal{M}_{loc}$ and $C\in\mathcal{V}^p$ such that
\[
X_t=X_T-\int^T_t\,dC_s-\int^T_t H_s\,dB_s,\quad t\in[0,T].
\]
Set $X'_t=X_t$, $t\in[0,T)$, $X'_T=\xi$. Then  for some $A',K'\in\mathcal{V}^{+,p}$ we have that
\[
X'_t=\xi+\int^T_t f(s,X'_s,H_s)\,ds+\int^T_t\,dV_s+\int^T_t\,dK'_s-\int^T_t\,dA'_s-\int^T_t H_s\,dB_s,\quad t\in[0,T].
\]
Let $(\hat{X}^n,\hat{H}^n,\hat{A}^n)$ be a solution of the  $\mathrm{\overline{R}BSDE}$
\begin{align}\label{Tw4.1.0}
\nonumber
\hat{X}^n_t&=\xi+\int^T_t f(s,\hat{X}^n_s,\hat{H}^n_s)\,ds+\int^T_t\,dV_s+\int^T_t\,dK'_s-\int^T_t\,d\hat{A}^n_s-\int^T_t \hat{H}^n_s\,dB_s\\
&\quad+\sum_{t\le\sigma_{n,i}<T}(\hat{X}^n_{\sigma_{n,i}+}
+\Delta^+V_{\sigma_{n,i}}-L_{\sigma_{n,i}})^-,\quad t\in[0,T],
\end{align}
with upper barrier $U$, such that if $p>1$, then $\hat{X}^n\in\mathcal{S}^p$, $\hat{H}^n\in\mathcal{H}^p$, $\hat{A}^n\in\mathcal{V}^{+,p}$, and if $p=1$, then $\hat{X}^n$ is of class \textnormal{(D)}, $\hat{H}^n\in\mathcal{H}^q$, $q\in(0,1)$, $\hat{A}^n\in\mathcal{V}^{+,1}$. The existence of such solution  follows from \cite[Theorem 3.20]{KRzS}. Note that $(X',H,A')$ is a solution of $\rm{\overline{R}}$BSDE$(\xi,f+dV+dK',X')$.
Since $X'\le U$, by Proposition \ref{proposition2.2} and Lemma \ref{uw2.0}, $\hat{X}^n\ge X'\ge L$. Thanks to this, we may rewrite (\ref{Tw4.1.0}) in the  form
\begin{align*}
\hat{X}^n_t&=\xi+\int^T_t f(s,\hat{X}^n_s,\hat{H}^n_s)\,ds+\int^T_t\,dV_s+\int^T_t\,dK'_s
-\int^T_t\,d\hat{A}^n_s+n\int^T_t(\hat{X}^n_s-L_s)^-\,ds\\
&\quad+\sum_{t\le\sigma_{n,i}<T}(\hat{X}^n_{\sigma_{n,i}+}+\Delta^+V_{\sigma_{n,i}}-L_{\sigma_{n,i}})^--\int^T_t \hat{H}^n_s\,dB_s,\quad t\in[0,T].
\end{align*}
By Proposition \ref{proposition2.2} and Lemma \ref{uw2.0} again,
\begin{equation}\label{Tw4.1.1.0}
\hat{X}^n\ge \bar{Y}^n,
\end{equation}
and by Proposition \ref{stw.4.1},
\begin{equation}\label{Tw4.1.1}
d\hat{A}^n\ge d\bar{A}^n,\quad n\ge 1.
\end{equation}
Moreover,
\begin{align}\label{Tw4.1.2}
\nonumber
(\hat{X}^n_{\sigma_{n,i}+}+\Delta^+V_{\sigma_{n,i}}-L_{\sigma_{n,i}})^-&\le (X'_{\sigma_{n,i}+}+\Delta^+V_{\sigma_{n,i}}-L_{\sigma_{n,i}})^-\\\nonumber
&=(\Delta^+X'_{\sigma_{n,i}}+\Delta^+V_{\sigma_{n,i}}+X'_{\sigma_{n,i}}
-L_{\sigma_{n,i}})^-\nonumber\\
&\le(\Delta^+X_{\sigma_{n,i}}+\Delta^+V_{\sigma_{n,i}})^- \nonumber \\
&\le\Delta^+|C|_{\sigma_{n,i}}+\Delta^+|V|_{\sigma_{n,i}}.
\end{align}
Let $(\tilde{X},\tilde{H},\tilde{A})$ be a solution of the following $\mathrm{\overline{R}BSDE}$
\begin{align*}
\tilde{X}_t&=\xi+\int^T_t f(s,\tilde{X}_s,\tilde{H}_s)\,ds+\int^T_t\,dV_s
+\int^T_t\,dK'_s-\int^T_t\,d\tilde{A}_s+n\int^T_t(\tilde{X}_s-L_s)^-\,ds\\
&\quad+\int^T_t\,d|C|_s+\int^T_t\,d|V|_s-\int^T_t \tilde{H}_s\,dB_s,\quad t\in[0,T],
\end{align*}
with upper barrier $U$, such that if $p>1$, then $\tilde{X}\in\mathcal{S}^p$, $\tilde{H}\in\mathcal{H}^p$, $\tilde{A}\in\mathcal{V}^{+,p}$, and if $p=1$, then $\tilde{X}$ is of class \textnormal{(D)}, $\tilde{H}\in\mathcal{H}^q$, $q\in(0,1)$, $\tilde{A}\in\mathcal{V}^{+,1}$. The existence of the solution follows from  \cite[Theorem 3.20]{KRzS}. The triple $(\tilde{X},\tilde{H},\tilde{A})$ does not depend on $n$, because by Proposition \ref{proposition2.2} and Lemma \ref{uw2.0}, $\tilde{X}\ge\hat{X}^n$, so the term involving $n$ on the right-hand side of the above equation equals zero. By the last inequality and (\ref{Tw4.1.1.0}),
\begin{equation}\label{Tw4.1.1.3}
\tilde{X}\ge \bar{Y}^n.
\end{equation}
By Proposition \ref{stw.4.1}, $d\hat{A}^n\le d\tilde{A}$, which by (\ref{Tw4.1.1}) implies that
\begin{equation}\label{Tw4.1.4}
d\bar{A}^n\le d\tilde{A}.
\end{equation}
By Proposition \ref{proposition2.2} and Lemma \ref{uw2.0},
\begin{equation}\label{wt.1}
\bar{Y}^n\le \bar{Y}^{n+1}.
\end{equation}
From this, (\ref{Tw4.1.1.3}), (\ref{Tw4.1.4}) and  \cite[Lemma 4.2]{kl}, if $p>1$ , then
\begin{align}\label{Tw.4.1}
\nonumber
E(K^n_T)^p+E\Big(\int^T_0|\bar{Z}^n_s|^2\,ds\Big)^{\frac{p}{2}}&\le CE\Big(\sup_{t\le T}(|Y^1_t|^p+|\tilde{X}_t|^p)+\Big(\int^T_0\,d|V|_s\Big)^p\\\nonumber
&\quad+\Big(\int^T_0|f^-(s,\tilde{X}_s,0)|\,ds\Big)^p+\Big(\int^T_0\tilde{X}^+_s\,ds\Big)^p\\
&\quad+\Big(\int^T_0|f(s,0,0)|\,ds\Big)^p+\Big(\int^T_0\,d\tilde{A}_s\Big)^p\Big),
\end{align}
and if $p=1$, then for every $q\in (0,1)$,
\begin{align}\label{Tw.4.3}
\nonumber
E\Big(\int^T_0|\bar{Z}^n_s|^2\,ds\Big)^{{q}/{2}}&\le CE\Big(\sup_{t\le T}(|Y^1_t|^q+|\tilde{X}_t|^q)+\Big(\int^T_0|f(s,0,0)|\,ds\Big)^q\\\nonumber
&\quad+\Big(\int^T_0|f^-(s,\tilde{X}_s,0)|\,ds\Big)^q+\Big(\int^T_0\tilde{X}^+_s\,ds\Big)^q\\
&\quad+\Big(\int^T_0\,d|V|_s\Big)^q+\Big(\int^T_0\,d\tilde{A}_s\Big)^q\Big).
\end{align}
%Step 3.
Now we will apply Theorem \ref{tw3.2} to (\ref{eq4.1}). We know that $\bar{Y}^n$ is of class (D), $\bar{Z}^n\in\mathcal{H}$, $K^n\in\mathcal{V}^+$, $\bar{A}^n\in\mathcal{V}^+$ and $t\mapsto f(t,\bar{Y}^n_t,\bar{Z}^n_t)\in \mathbb{L}^1(0,T)$ and $V$ is a finite variation process. By Proposition \ref{stw.4.1},  $d\bar A^n\le d\bar{A}^{n+1}$, $n\in\mathbb{N}$. Let $Y'_t=\sup_{n\ge 1}\bar{Y}^n_t$, $A_t=\lim_{n\rightarrow\infty}\bar{A}^n_t$, $t\in[0,T]$ and $D^n:=\bar{A}^n-V$. We will check assumptions (a)--(f) of Theorem \ref{tw3.2}.
\begin{enumerate}[(a)]
\item We have shown that $d\bar{A}^n\le d\bar{A}^{n+1}$ and $d\bar{A}^n\le d\tilde{A}$, $n\ge 1$. Hence $dD^n\le dD^{n+1}$ and $\sup_{n\ge 1}E|D^n|_T<\infty$.
\item Let $\tau,\sigma\in\Gamma$ be stopping times such that $\sigma\le\tau$. By (\ref{Tw4.1.4}),
\begin{align*}
&\liminf_{n\rightarrow\infty}\Big(\int^{\tau}_{\sigma}(Y'_s-\bar{Y}^n_s)\,d(K^n_s-D^n_s)+\sum_{\sigma\le s<\tau}(Y'_s-\bar{Y}^n_s)\Delta^+(K^n_s-D^n_s)\Big)\\
&\quad\ge-\lim_{n\rightarrow\infty}\Big(\int^{\tau}_{\sigma}(Y'_s-\bar{Y}^n_s)\,(d\tilde{A}_s-dV_s)+\sum_{\sigma\le s<\tau}(Y'_s-\bar{Y}^n_s)(\Delta^+\bar{A}^n_s-\Delta^+V_s)\Big).
\end{align*}
By the Lebesgue dominated convergence theorem,
\[
\lim_{n\rightarrow\infty}\Big(\int^{\tau}_{\sigma}(Y'_s-\bar{Y}^n_s)\,(d\tilde{A}_s-dV_s)+\sum_{\sigma\le s<\tau}(Y'_s-\bar{Y}^n_s)(\Delta^+\tilde{A}_s-\Delta^+V_s)\Big)=0.
\]
Therefore $\liminf_{n\rightarrow
\infty}\int^{\tau}_{\sigma}(Y'_s-\bar{Y}^n_s)\,d(K^n_s-D^n_s)\ge 0$.
\item It is easy to see that $\Delta^-K^n_t=0$, $n\in\mathbb{N}$, $t\in[0,T]$.
\item Let $\bar{y}=Y^1$ and $\underline{y}=\tilde{X}$. Then $\bar{y},\underline{y}\in\mathcal{V}^1+\mathcal{M}_{loc}$, $\bar{y},\, \underline{y}$ are of class (D) and by (H6), (H6*) and Proposition \ref{prop2.5},
\[
E\int^T_0 f^+(s,\bar{y}_s,0)\,ds+E\int^T_0 f^-(s,\underline{y}_s,0)\,ds<\infty.
\]
By (\ref{Tw4.1.1.3}), $\bar{y}_t\le \bar{Y}^n_t\le\underline{y}_t$, $t\in[0,T]$.
\item It follows from (H3).
\item By the definition of $Y'$, $\bar{Y}^n_t\nearrow Y'_t$, $t\in[0,T]$.
\end{enumerate}
By Theorem \ref{tw3.2}, $Y'$ is regulated and there exist processes $K\in\mathcal{V}^+$ and $Z'\in \mathcal{H}$ such that
\begin{equation}
\label{eq.main}
Y'_t=\xi+\int^T_tf(s,Y'_s,Z'_s)\,ds+\int^T_t\,dV_s+\int^T_t\,dK_s-\int^T_t\,dA_s-\int^T_t Z'_s\,dB_s,\quad t\in[0,T].
\end{equation}
Moreover $\bar{Z}^n\rightarrow Z'$ in the sense of (\ref{eq2.2}).  This when combined with (\ref{Tw.4.1}) and (\ref{Tw.4.3}) implies that if $p>1$, then $Z'\in \mathcal{H}^p$ and (\ref{Tw4.1.0.0}) is satisfied and if $p=1$, then $Z'\in \mathcal{H}^q,\, q\in (0,1)$, (\ref{Tw4.1.0.1}) holds, and there exists a chain $\{\tau_k\}\subset \Gamma$ such that (\ref{eq.kled.ert}) is satisfied.

{\bf Step 3.} We will show that $EK^p_T+EA^p_T<\infty$. The desired integrability of $A$ follows from the integrability of $\tilde{A}$ and (\ref{Tw4.1.4}). To prove that $EK^p_T<\infty$, we show that
\begin{equation}\label{Tw.4.4}
\sup_{n\ge 1}E\Big(\int^T_0 |f(s,\bar{Y}^n_s,\bar{Z}^n_s)|\,ds\Big)^p+E\Big(\int^T_0 |f(s,Y'_s,Z'_s)|\,ds\Big)^p<\infty.
\end{equation}
If $p>1$, then by (H1), (H2), (\ref{Tw4.1.1.3}) and (\ref{wt.1})
\begin{align*}
&E\Big(\int^T_0|f(s,\bar{Y}^n_s,\bar{Z}^n_s)|\,ds\Big)^p\le C_p\Big(E\Big(\int^T_0|f(s,\tilde{X}_s,0)|\,ds\Big)^p\\
&\quad+E\Big(\int^T_0|f(s,Y^1_s,0)|\,ds\Big)^p
+E\Big(\int^T_0|\bar{Z}^n_s|^2\,ds\Big)^{{p}/{2}}\Big).
\end{align*}
If $p=1$, then by (Z),
\begin{align*}
E\int^T_0 |f(s,\bar{Y}^n_s,\bar{Z}^n_s)|\,ds&\le \gamma E\int^T_0 (g_s+|\bar{Y}^n_s|+|\bar{Z}^n_s|)^{\alpha}\,ds+E\int^T_0 |f(s,\bar{Y}^n_s,0)|\,ds.
\end{align*}
By $\mathrm{H\ddot{o}lder's}$ inequality, (H2), (\ref{Tw4.1.1.3}) and (\ref{wt.1}),
\begin{align*}
&E\int^T_0 (g_s+|\bar{Y}^n_s|+|\bar{Z}^n_s|)^{\alpha}\,ds+E\int^T_0 |f(s,\bar{Y}^n_s,0)|\,ds\\
&\le C\Big\{E\Big(\int^T_0|\bar{Z}^n_s|^2\,ds\Big)^{{\alpha}/{2}}+E\int^T_0 (g_s+|\tilde{X}_s|+|Y^1_s|)^{\alpha}\,ds\\
&\quad+E\int^T_0 |f(s,Y^1_s,0)|+|f(s,\tilde{X}_s,0)|\,ds\Big\}.
\end{align*}
Applying Fatou's lemma and using  (\ref{Tw.4.1}) and  (\ref{Tw.4.3}) we get (\ref{Tw.4.4}). The desired integrability of $K$ follows from (\ref{Tw.4.4}) and the integrability  of $Y',Z'$  and $A,V$.

{\bf Step 4.} We show that the minimality condition for $A$ is satisfied, i.e.
\begin{equation}\label{wt.2}
\int^T_0(U_{t-}-Y'_{t-})\,dA^*_t+\sum_{t<T}(U_t-Y'_t)\Delta^+A_t=0.
\end{equation}
Since the triple $(\bar{Y}^n,\bar{Z}^n,\bar{A}^n)$ is a solution of (\ref{eq4.1}), we have
\begin{equation}\label{wt.3}
\int^T_0(U_{t-}-\bar{Y}^n_{t-})\,dA^{n,*}_t+\sum_{t<T}(U_t-\bar{Y}^n_t)\Delta^+\bar{A}^n_t=0.
\end{equation}
By the Vitali--Hahn--Saks theorem, $d\bar{A}^n\nearrow dA$ in the variation norm, i.e.
\begin{equation}\label{wt.4}
\Delta^+\bar{A}^n_t\nearrow\Delta^+A_t,\quad\Delta^-\bar{A}^n_t\nearrow\Delta^-A_t,\quad|dA^{n,*,c}-dA^{*,c}|_{TV}\rightarrow 0.
\end{equation}
Letting $n\rightarrow\infty$ in the second term of (\ref{wt.3}) and applying  the Lebesgue dominated convergence theorem we obtain
\[
\sum_{t<T}(U_t-Y'_t)\Delta^+A_t=0.
\]
Since $|dA^{n,*,c}-dA^{*,c}|_{TV}\rightarrow 0$ and $0\le U_t-\bar{Y}^n_t\le U_t-Y^1_t$, using (\ref{wt.3}) and the Lebesgue dominated convergence theorem we get $\int^T_0(U_t-Y'_t)\,dA^{*,c}_t=0$. If $\Delta^-A^*_t=0$, then $(U_{t-}-Y'_{t-})\Delta^-A^*_t=0$. If $\Delta^-A_t^*>0$, then by (\ref{wt.4}) there exists $N\in\mathbb{N}$ such that $\Delta^-A_t^{n,*}>0$ for $n\ge N$. By this and (\ref{wt.3}), $\bar{Y}^n_{t-}=U_{t-}$ for $n\ge N$. By Proposition \ref{proposition2.2} and Lemma \ref{uw2.0}, $Y'_{t-}\ge \bar{Y}^n_{t-}=U_{t-}$\,, so $Y'_{t-}=U_{t-}$. Therefore
\[
\sum_{t\le T}(U_{t-}-Y'_{t-})\Delta^-A^*_t=0.
\]

{\bf Step 5.} We will show that $Y'\ge L$. By (\ref{Tw4.1.1.3}), (\ref{Tw4.1.4}) and (\ref{Tw.4.4}), $\sup_{n\ge 1}EK^n_T<\infty$, so   $\Big\{n\int^T_0(\bar{Y}^n_s-L_s)^-\,ds\Big\}$ is bounded in $\mathbb{L}^1(\Omega)$. Therefore, passing to a subsequence if necessary, we may assume that there exists a dense countable subset $Q\subset [0,T]$ such that for $P$-a.e. $\omega\in\Omega$, $(\bar{Y}^n_t-L_t)^-\rightarrow 0$ for $t\in Q$. Consequently, $Y'_t\ge L_t$ for $t\in Q$. Hence $Y'_{t+}\ge L_{t+},\, t\in [0,T]$. We will show that $Y'_t\ge L_t$ for every $t\in[0,T)$. Let $t\in[0,T)$. Assume that $\Delta^+(L_t+V_t)\ge 0$. If $\Delta^+A_t>0$, then $Y'_t=U_t$, so obviously $Y'_t\ge L_t$. In case $\Delta^+A_t=0$,  we have  $\Delta^+Y'_t=-\Delta^+V_t-\Delta^+K_t$. Therefore
\begin{align*}
Y'_t+V_t=-(\Delta^+V_t+\Delta^+Y'_t)+Y'_{t+}+V_{t+}\ge L_{t+}+V_{t+}\ge L_t+V_t,
\end{align*}
so $Y'_t\ge L_t$. Assume now that $\Delta^+(L_t+V_t)<0$. If $\Delta^+A_t>0$, then $Y'_t=U_t$, so $Y'_t\ge L_t$. If $\Delta^+A_t=0$, then by (\ref{wt.4}), $\Delta^+\bar{A}^n_t=0$, $n\ge 1$. Since $\Delta^+(L_t+V_t)<0$, $t\in\bigcup_i[[\sigma_{n,i}]]$ for sufficiently large $n$. Hence  $\Delta^+K^n_t=(\bar{Y}^n_{t+}+\Delta^+V_t-L_t)^-$. By this and (\ref{eq4.1}),
\[
\Delta^+\bar{Y}^n_t=-\Delta^+V_t-(\bar{Y}^n_{t+}+\Delta^+V_t-L_t)^-.
\]
Suppose that  $\bar{Y}^n_t<L_t$. Then
\begin{align*}
\bar{Y}^n_{t+}-L_t+\Delta^+V_t<\bar{Y}^n_{t+}-\bar{Y}^n_t+\Delta^+V_t=-(\bar{Y}^n_{t+}+\Delta^+V_t-L_t)^-.
\end{align*}
Consequently, $\bar{Y}^n_{t+}+\Delta^+V_t-L_t<-(\bar{Y}^n_{t+}+\Delta^+V_t-L_t)^-$, which is a contradiction. Thus $\bar{Y}^n_t\ge L_t$, so $Y'_t\ge L_t$. Therefore
\[
Y'_t\ge L_t\mathbf{1}_{\{t<T\}}+\xi\mathbf{1}_{\{t=T\}},\quad t\in[0,T].
\]

{\bf Step 6.} We will show the minimality condition for $K$, i.e. we show that
\begin{equation}\label{wt.5}
\int^T_0(Y'_{r-}-L_{r-})\,dK^*_r+\sum_{r<T}(Y'_r-L_r)\Delta^+K_r=0.
\end{equation}
By (\ref{eq.main}), (\ref{Tw.4.4}) and the integrability properties of $Y',V$ and $A$, the  process 
\[Y'+\int^{\cdot}_0 f(s,Y'_s,Z'_s)\,ds-V+A\] is a supermartingale which majorizes the process $L\mathbf{1}_{\{\cdot<T\}}+\xi\mathbf{1}_{\{\cdot=T\}}+\int^{\cdot}_0 f(s,Y'_s,Z'_s)\,ds-V+A$. Hence
\begin{equation}\label{Tw.4.5}
Y'_t\ge \esssup_{\tau\in\Gamma_t}E\Big(\int^{\tau}_t f(s,Y'_s,Z'_s)\,ds+\int^{\tau}_t\,dV_s-\int^{\tau}_t\,dA_s+L_{\tau}\mathbf{1}_{\{\tau<T\}}
+\xi\mathbf{1}_{\{\tau=T\}}|\mathcal{F}_t\Big).
\end{equation}
Let $p>1$. By \cite[Proposition 3.13]{KRzS}, Step 1 and the definition of $L^n$, for $t\in[0,T]$ we have
\begin{align}\label{wt.6}
\nonumber
\bar{Y}^n_t&=\esssup_{\tau\in\Gamma_t}E\Big(\int^{\tau}_t f(s,\bar{Y}^n_s,\bar{Z}^n_s)\,ds+\int^{\tau}_t\,dV_s-\int^{\tau}_t\,d\bar{A}^n_s
+L^n_{\tau}\mathbf{1}_{\{\tau<T\}}+\xi\mathbf{1}_{\{\tau=T\}}|\mathcal{F}_t\Big)\\\nonumber
&\le \esssup_{\tau\in\Gamma_t}E\Big(\int^{\tau}_t f(s,\bar{Y}^n_s,\bar{Z}^n_s)\,ds+\int^{\tau}_t\,dV_s\\
&\qquad\qquad\qquad\qquad\qquad-\int^{\tau}_t\,d\bar{A}^n_s+L_{\tau}\mathbf{1}_{\{\tau<T\}}+\xi\mathbf{1}_{\{\tau=T\}}|\mathcal{F}_t\Big).
\end{align}
Observe that by (\ref{Tw4.1.0.0}), (\ref{Tw.4.4}) and the assumptions on $f$,
\begin{equation}\label{wt.7}
E\int^T_0|f(r,\bar{Y}^n_r,\bar{Z}^n_r)-f(r,Y'_r,Z'_r)|\,dr\rightarrow 0.
\end{equation}
By (\ref{wt.4}), (\ref{wt.6}), (\ref{wt.7}) and \cite[Lemma 3.19]{KRzS},
\begin{equation*}
Y'_t\le \esssup_{\tau\in\Gamma_t}E\Big(\int^{\tau}_t f(s,Y'_s,Z'_s)\,ds+\int^{\tau}_t\,dV_s-\int^{\tau}_t\,dA_s+L_{\tau}\mathbf{1}_{\tau<T}+\xi\mathbf{1}_{\tau=T}|\mathcal{F}_t\Big).
\end{equation*}
This when combined with (\ref{Tw.4.5}) gives
\begin{equation*}
Y'_t=\esssup_{\tau\in\Gamma_t}E\Big(\int^{\tau}_t f(s,Y'_s,Z'_s)\,ds+\int^{\tau}_t\,dV_s-\int^{\tau}_t\,dA_s+L_{\tau}\mathbf{1}_{\tau<T}+\xi\mathbf{1}_{\tau=T}|\mathcal{F}_t\Big).
\end{equation*}

Let $p=1$. Since $Y^1\le \bar{Y}^n\le Y'$, $n\ge 1$, using (H2) we get
\[
f(t,Y'_t,0)\le f(t,\bar{Y}^n_t,0)\le f(t,Y^1_t,0),\quad t\in[0,T].
\]
Define $\sigma_k=\inf\{t\ge 0:\int^t_0|f(r,Y^1_r,0)|+|f(r,Y'_r,0)|\,dr\ge k\}\wedge T$.
It is clear  that $\{\sigma_k\}$ is a chain. We may assume that $\sigma_k=\tau_k$. Observe that by (\ref{Tw4.1.0.1}), the definition of $\sigma_k$ and the assumptions on $f$,
\begin{equation}\label{wt.9}
E\int^{\tau_k}_0|f(r,\bar{Y}^n_r,\bar{Z}^n_r)-f(r,Y_r,Z_r)|\,dr\rightarrow 0.
\end{equation}
By \cite[Proposition 3.13]{KRzS}, Step 1, (f) and the definition of $L^n$, for $t\in[0,\tau_k]$ we have
\begin{align*}
\nonumber
\bar{Y}^n_t&=\esssup_{\tau \in\Gamma_t,\, \tau_k\ge\tau}E\Big(\int^{\tau}_t f(s,\bar{Y}^n_s,\bar{Z}^n_s)\,ds+\int^{\tau}_t\,dV_s
-\int^{\tau}_t\,d\bar{A}^n_s+L^n_{\tau}\mathbf{1}_{\{\tau<\tau_k\}}
+\bar{Y}^n_{\tau_k}\mathbf{1}_{\{\tau=\tau_k\}}|\mathcal{F}_t\Big)\\
&\le \esssup_{\tau \in\Gamma_t, \tau_k\ge\tau}E\Big(\int^{\tau}_t f(s,\bar{Y}^n_s,\bar{Z}^n_s)\,ds+\int^{\tau}_t\,dV_s-\int^{\tau}_t\,d\bar{A}^n_s
+L_{\tau}\mathbf{1}_{\{\tau<\tau_k\}}+Y'_{\tau_k}\mathbf{1}_{\{\tau=\tau_k\}}|\mathcal{F}_t\Big).
\end{align*}
By this, (\ref{wt.4}), (\ref{wt.9}) and \cite[Lemma 3.19]{KRzS},
\begin{equation*}
Y'_t\le \esssup_{\tau \in\Gamma_t,\, \tau_k\ge\tau}E\Big(\int^{\tau}_t f(s,Y'_s,Z'_s)\,ds+\int^{\tau}_t\,dV_s-\int^{\tau}_t\,dA_s+L_{\tau}\mathbf{1}_{\{\tau<\tau_k\}}
+Y'_{\tau_k}\mathbf{1}_{\{\tau=\tau_k\}}|\mathcal{F}_t\Big).
\end{equation*}
This when combined with (\ref{Tw.4.5}) gives
\begin{equation*}
Y'_t=\esssup_{\tau \in\Gamma_t,\, \tau_k\ge\tau}E\Big(\int^{\tau}_t f(s,Y'_s,Z'_s)\,ds+\int^{\tau}_t\,dV_s-\int^{\tau}_t\,dA_s+L_{\tau}\mathbf{1}_{\tau<\tau_k}+Y'_{\tau_k}\mathbf{1}_{\tau=\tau_k}|\mathcal{F}_t\Big).
\end{equation*}
By \cite[Corollary 3.11]{KRzS} we have (\ref{wt.5}) satisfied on $[0,\tau_k]$, and since $\{\tau_k\}$ is a chain, we have it also on $[0,T]$.

{\bf Step 7.} We will show that $(Y',Z',K-A)=(Y,Z,R)$.
Put $R'=K-A$. Obviously $dR^{'+}\le dK$ and  $dR^{'-}\le dA$, so by (\ref{wt.2}) and (\ref{wt.5}),
\[
\int^T_0 (Y_{r-}-L_{r-})\,dR'^{+,*}_r+\sum_{r<T}(Y_r-L_r)\Delta^+R'^+_r=0
\]
and
\[
\int^T_0 (U_{r-}-Y_{r-})\,dR'^{-,*}_r+\sum_{r<T}(U_r-Y_r)\Delta^+R'^-_r=0.
\]
Consequently, the triple $(Y',Z',R')$ is a solution  of RBSDE$(\xi,f+dV,L,U)$ such that $Y'\in\mathcal{S}^p$, $Z'\in\mathcal{H}^p$, $R'\in\mathcal{S}^p$  in case $p>1$, and in case $p=1$
$Y'$ is of class \textnormal{(D)}, $Z'\in\mathcal{H}^q$, $q\in(0,1)$, $R'\in\mathcal{V}^1$ in case $p=1$. Hence,
by Proposition \ref{proposition2.2} and Lemma \ref{uw2.0},  $(Y',Z',K-A)=(Y,Z,R)$.

{\bf Step 8.} We  will show that if $\Delta^-R^+=0$, then  (\ref{Tw4.1.0.0}), (\ref{Tw4.1.0.1}) hold with $\gamma=2$ and $|\bar Y^n-Y|_p\rightarrow0$.
Let $\bar R^n=K^n-\bar A^n$. By  \cite[Corollary A.5]{KRzS}, (H1) and (H2),
\begin{align}\label{czwartek.7fp}
\nonumber
\int^T_0|Z_r-\bar Z^n_r|^2\,dr&\le 2\lambda\int^T_0|Y_r-\bar Y^n_r||Z_r-\bar Z^n_r|\,dr+2\int^T_0(Y_{r-}-\bar Y^n_{r-})\,d(R-\bar R^n)_r^{*}\\\nonumber
&\quad+2\sum_{0\le t<T}(Y_t-\bar Y^n_t)\Delta^+(R_t-\bar R^n_t)\\
&\quad+\sup_{0\le t\le T}\Big|\int^T_t(Y_r-\bar Y^n_r)(Z_r-\bar Z^n_r)\,dB_r\Big|
\end{align}
By the the minimality condition and the assumption that $\Delta^-R^+=0$
\begin{equation*}
\sum_{0\le t<T}(Y_t-\bar Y^n_t)\Delta^+(R_t-\bar R^n_t)\le\sum_{0\le t<T}(Y_t-\bar Y^n_t)\Delta^+R_t\le \sum_{0\le t<T}(Y_t-\bar Y^n_t)\Delta^+R^+_t
\end{equation*}
and
\[
\int^T_0(Y_{r-}-\bar Y^n_{r-})\,d(R-\bar R^n)_r^{*}\le \int^T_0(Y_{r-}-\bar Y^n_{r-})\,dR^{+,*}_r=\int^T_0(Y_{r}-\bar Y^n_{r})\,dR^{+,c}_r.
\]
By the above, (\ref{czwartek.7fp}) and the Burkholder--Davis--Gundy inequality,
\begin{equation}\label{czwartek.8fp}
E\Big(\int^T_0|Z_r-\bar Z^n_r|^2\,dr\Big)^{{p}/{2}}\le C(E\sup_{0\le t\le T} |Y_t-\bar Y^n_t|^p+(E\sup_{0\le t\le T}|Y_t-\bar Y^n_t|^p)^{{1}/{2}}(E|R|^p_T)^{{1},{2}}).
\end{equation}
Furthermore,  $\Delta^-Y_t=\Delta^-R^-_t$ and $\Delta ^-\bar Y^n_t=\Delta^-\bar A^n_t$. Hence,
if $\Delta^- \bar Y^n_t=0$, $n\ge 0$, then $\Delta^-Y_t=0$. Otherwise, i.e. if $\Delta^- \bar Y^n_t=\Delta^-\bar A^n_t>0$ for some $n\ge 1$, then $\Delta^-\bar Y^n_t>0$, $n\ge N$
($\{\Delta^-\bar A^n_t\}$ is nondecreasing), which implies that
 $\bar Y^n_{t-}=U_{t-}$ for $n\ge N$. Since $U_{t-}=\bar Y^n_{t-}\le Y_{t-}\le U_{t-}$, this implies that $\bar Y^n_{t-}=Y_{t-},\, n\ge N$. Thus, in both cases,  $\Delta^-\bar Y^n_t\rightarrow \Delta^-Y_t,\, t\in (0,T]$. Moreover, by the construction of $\bar Y^n$,  $\Delta^+\bar Y^n_t\rightarrow \Delta^+Y_t,\, t\in [0,T)$. Consequently, by the generalized Dini theorem,
 $\sup_{0\le t\le T}|Y_t-\bar Y^n_t|\rightarrow 0$ as $n\rightarrow\infty$. Therefore, by (\ref{Tw4.1.1.3})  and the Lebesgue dominated convergence theorem,
$|Y-\bar Y^n|_p\rightarrow 0$.
From this and (\ref{czwartek.8fp}) we deduce that (\ref{Tw4.1.0.0}) and  (\ref{Tw4.1.0.1})  hold with $\gamma=2$.
\end{proof}

Analogously to (\ref{eq4.1}), we define  $(\underline{Y}^n,\underline{Z}^n,\underline{A}^n)$ as a solution of $\mathrm{\underline{R}BSDE}$
\begin{align}\label{r.1}
\nonumber
\underline{Y}^n_t&=\xi+\int^T_t f(r,\underline{Y}^n_r,\underline{Z}^n_r)\,dr+\int^T_t\,dV_r+\int^T_t\,d\underline{K}^n_r-\int^T_t \underline{Z}^n_r\,dB_r\\
&\quad-n\int^T_t(\underline{Y}^n_r-U_r)^-\,dr-\sum_{t\le\tau_{n,i}<T}(\underline{Y}^n_{\tau_{n,i}+}+\Delta^+V_{\tau_{n,i}}-U_{\tau_{n,i}})^+,
\end{align}
with lower barrier $L$, such that if $p>1$, then $\underline{Y}^n\in\mathcal{S}^p$, $\underline{Z}^n\in\mathcal{H}^p$, $\underline{K}^n\in\mathcal{V}^{+,p}$, and if $p=1$, then $\underline{Y}^n$ is of class (D), $\underline{Y}^n\in\mathcal{S}^q$, $\underline{Z}^n\in\mathcal{H}^q$, $q\in(0,1)$, $\underline{K}^n\in\mathcal{V}^{+,1}$. Now $\{\{\tau_{n,i}\}\}$ is defined as follows:  we set $\tau_{1,0}=0$ and then
\[
\tau_{1,i}=\inf\{t>\tau_{1,i-1};\,\Delta^+U_t>1\,\,\mathrm{or}\,\,\Delta^+V_t>1\}\wedge T,\,i=1,\dots,k_1
\]
for some $k_1\in\N$. Next, for $n\in\N$ and given array $\{\{\tau_{n,i}\}\}$, we set $\tilde{\tau}_{n+1,0}=0$,
\[
\tilde{\tau}_{n+1,i}=\inf\{t>\tilde{\tau}_{n+1,i-1}:\Delta^+U_t>1/(n+1)
\,\,\mathrm{or}\,\,\Delta^+V_t>1/(n+1)\}\wedge T,\,i\ge 1.
\]
Let $j_{n+1}$ be chosen so that $P(\tilde{\tau}_{n+1,j_{n+1}}<T)\le\frac1n$. We put
\[\tau_{n+1,i}=\tilde{\tau}_{n+1,i},\quad i=1,\dots,j_{n+1},\quad\tau_{n+1,i+j_{n+1}}
=\tilde{\tau}_{n+1,j_{n+1}}\vee \tau_{n,i},\quad i=1,....,k_n,\]
$k_{n+1}=j_{n+1}+k_{n}+1$. Finally, we put $\tau_{n+1,k_{n+1}}=T$.

\begin{theorem}\label{czw.3}
Let $(\underline{Y}^n,\underline{Z}^n,\underline{A}^n)$, $n\in\mathbb{N}$ be defined by \mbox{\rm(\ref{r.1})}.
\begin{enumerate}[\rm(i)]
\item Assume that $p>1$ and \textnormal{(H1)--(H6)} are satisfied. Then $\underline{Y}^n_t\searrow Y_t$, $t\in[0,T]$, and for every $\gamma\in[1,2)$,
\[
E\Big(\int^T_0|\underline{Z}^n_r-Z_r|^{\gamma}\,dr\Big)^{{p}/{2}}\rightarrow 0,
\]
where $(Y,Z,R)$ is the unique solution of \textnormal{RBSDE}($\xi$,$f+dV$,$L$,$U$) such that $Y\in\mathcal{S}^p$, $Z\in\mathcal{H}^p$, $R\in\mathcal{S}^p$. Moreover, if $\Delta^-R^-_t=0$ for $t\in(0,T]$, then the above convergence also hold with $\gamma=2$ and $|\underline Y^n-Y|_p\rightarrow 0$.

\item Assume that  $p=1$ and  \textnormal{(H1)--(H5), (H6*)}, \textnormal{(Z)} are satisfied. Then $\underline{Y}^n_t\searrow Y_t$, $t\in[0,T]$, and for all $\gamma\in[1,2)$ and $r\in(0,1)$,
\[
E\Big(\int^T_0|\underline{Z}^n_r-Z_r|^{\gamma}\,dr\Big)^{{r}/{2}}\rightarrow 0,
\]
where $(Y,Z,R)$ is the unique solution of \textnormal{RBSDE}($\xi$,$f+dV$,$L$,$U$) such that $Y$ is of class \textnormal{(D)}, $Y\in\mathcal{S}^q$, $Z\in\mathcal{H}^q$, $R\in\mathcal{V}^1$, $q\in(0,1)$.
Moreover, if $\Delta^-R^-_t=0$ for $t\in(0,T]$, then the above convergence also hold with $\gamma=2$ and $|\underline Y^n-Y|_1\rightarrow 0$.
\end{enumerate}
\end{theorem}
\begin{proof}
By Definition \ref{r.2}, $(-\underline{Y}^n,-\underline{Z}^n,\underline{K}^n)$ is a solution of $\mathrm{\overline{R}BSDE}$  of the  form
\begin{align}
\label{r.3}
-\underline{Y}^n_t&=-\xi-\int^T_t f(r,-\underline{Y}^n_r,-\underline{Z}^n_r)\,dr-\int^T_t\,dV_r
-\int^T_t\,d\underline{K}^n_r+\int^T_t \underline{Z}^n_r\,dB_r\nonumber\\
&\quad+n\int^T_t(\underline{Y}^n_r-U_r)^+\,dr
+\sum_{t\le\sigma_{n,i}<T}(\underline{Y}^n_{\sigma_{n,i}+}
+\Delta^+V_{\sigma_{n,i}}-U_{\sigma_{n,i}})^+
\end{align}
with upper barrier $-L$.
By Theorem \ref{r.6}, solutions of the above equation tend to the solution of RBSDE($-\xi$,$-\tilde{f}-dV$,$-U$,$-L$) with $\tilde{f}(t,y,z)=-f(t,-y,-z)$, from which the desired result follows.
\end{proof}

\subsubsection{Penalization method via BSDEs}

In this section we consider approximation of solutions of  RBSDE with two barriers by solutions of usual BSDEs. Let  $(Y^n,Z^n)$ be a solution of $\mathrm{BSDE}$ of the form
\begin{align}\label{r.4}
\nonumber
Y^n_t&=\xi+\int^T_t f(r,Y^n_r,Z^n_r)\,dr+\int^T_t\,dV_r-\int^T_t Z^n_r\,dB_r+n\int^T_t(Y^n_r-L_r)^-\,dr\\\nonumber
&\quad+\sum_{t\le\sigma_{n,i}<T}(Y^n_{\sigma_{n,i}+}+\Delta^+V_{\sigma_{n,i}}
-L_{\sigma_{n,i}})^--n\int^T_t(Y^n_r-U_r)^+\,dr\\
&\quad-\sum_{t\le\tau_{n,i}<T}(Y^n_{\tau_{n,i}+}+\Delta^+V_{\tau_{n,i}}-U_{\tau_{n,i}})^+
\end{align}
such that if $p>1$, then $Y^n\in\mathcal{S}^p$, $Z^n\in\mathcal{H}^p$, and if $p=1$, then $Y^n$ is of class (D), $Y^n\in\mathcal{S}^q$, $Z^n\in\mathcal{H}^q$, $q\in(0,1)$, where $\{\{\sigma_{n,i}\}\}$ and $\{\{\tau_{n,i}\}\}$ are defined by (\ref{eq4.1}) and (\ref{r.1}).  One can find a solution of (\ref{r.4}) inductively in the manner used to solve (\ref{eq4.1}) and (\ref{r.1}). More precisely, for fixed $n\in\N$ let  $k_{n}$ be the   number of stopping times $\{\sigma_{n,i}\}$ and $\{\tau_{n,i}\}$.
We put $m_n=2k_{n}$, $\gamma_{n,0}=0$, $\gamma_{n,1}=\sigma_{n,1}\wedge\tau_{n,1}$  and   $\gamma_{n,m}=\bar\sigma_{n,m}\wedge\bar\tau_{n,m}$, $m=2,\dots,m_n$, where
\[
\bar\sigma_{n,m}=\min\{\sigma_{n,i}:\sigma_{n,i}>\gamma_{n,m-1},\,i=1,\dots,k_{n}\}\wedge T \]
and
\[
\bar\tau_{n,m}=\min\{\tau_{n,i}:\tau_{n,i}>\gamma_{n,m-1},\,i=1,\dots,k_{n}\}\wedge T.
\]
Note that $\{\gamma_{n,m}\}$  are stopping times such that $\gamma_{n,m_n}=T$ and
\[
\bigcup^{k_n}_{i=1} [[\sigma_{n,i}]]\cup \bigcup^{k_{n}}_{i=1} [[\tau_{n,i}]]=\bigcup^{m_n}_{m=1} [[\gamma_{n,m}]].
\]
Moreover, for $m=1,\dots,{m_n}$, on  each interval $(\gamma_{n,m-1},\gamma_{n,m}]$,
the pair
$({Y}^n,{Z}^n)$ is a solution of the nonreflected  BSDEs of the form
\begin{align*}
Y^n_t&=L_{\gamma_{n,m}}\vee(Y^n_{\gamma_{n,m}+}+\Delta^+V_{\gamma_{n,m}})\wedge U_{\gamma_{n,m}}+\int^{\gamma_{n,m}}_t f(r,{Y}^n_r,{Z}^n_r)\,dr+\int^{\gamma_{n,m}}_t\,dV_r\\
&\quad-\int^{\gamma_{n,m}}_t {Z}^n_r\,dB_r+n\int^{\gamma_{n,m}}_t({Y}^n_r-L_r)^-\,dr -n\int^{\gamma_{n,m}}_t(Y^n_r-U_r)^+\,dr,\, t\in(\gamma_{n,m-1},\gamma_{n,m}],
\end{align*}
with  ${Y}^n_0=L_0\vee(\bar{Y}^n_{0+}+\Delta^+V_0)\wedge U_0$, $n\in\mathbb{N}$. Therefore,  to solve (\ref{eq4.1}), we divide $[0,T]$ into the finite number of intervals $[0,\gamma_{n,1}],\dots,(\gamma_{n,m_{n}-1},T]$ and we solve this equation on these intervals $(\gamma_{n,m-1},\gamma_{n,m}]$ inductively starting from the interval $(\gamma_{n,m_{n}-1},T]$.

\begin{theorem}
\label{th.r.1234}
Let $(Y^n,Z^n)$, $n\in\mathbb{N}$ be defined by \mbox{\rm(\ref{r.4})}.
\begin{enumerate}[\rm(i)]
\item Assume that $p>1$ and \textnormal{(H1)--(H6)} are satisfied. Then $Y^n_t\rightarrow Y_t$, $t\in[0,T]$, and for every $\gamma\in[1,2)$,
\begin{equation}
\label{eq.wewo1}
E\Big(\int^T_0|Z^n_r-Z_r|^{\gamma}\,dr\Big)^{{p}/{2}}\rightarrow 0,
\end{equation}
where $(Y,Z,R)$ is the unique solution of \textnormal{RBSDE}($\xi, f+dV, L, U$) such that $Y\in\mathcal{S}^p$, $Z\in\mathcal{H}^p$, $R\in\mathcal{S}^p$. Moreover, if $\Delta^-R_t=0$ for $t\in(0,T]$, then the above convergence also hold with $\gamma=2$ and $|Y^n-Y|_p\rightarrow 0$.

\item Assume that $p=1$ and \textnormal{(H1)--(H5), (H6*)}, \textnormal{(Z)} are satisfied. Then $Y^n_t\rightarrow Y_t$, $t\in[0,T]$, and for all $\gamma\in[1,2)$ and $r\in(0,1)$,
\begin{equation}
\label{eq.wewo2}
E\Big(\int^T_0|Z^n_r-Z_r|^\gamma\,dr\Big)^{{r}/{2}}\rightarrow 0,
\end{equation}
where $(Y,Z,R)$ is the unique solution of \textnormal{RBSDE}($\xi$,$f+dV$,$L$,$U$) such that $Y$ is of class \textnormal{(D)}, $Y\in\mathcal{S}^q$, $Z\in\mathcal{H}^q$, $R\in\mathcal{V}^1$, $q\in(0,1)$. Moreover, if $\Delta^-R_t=0$ for $t\in(0,T]$, then the above convergence also hold with $\gamma=2$ and  $| Y^n-Y|_1\rightarrow 0$
\end{enumerate}
\end{theorem}
\begin{proof}
Notice that (\ref{r.4}) one can written in the shorter form
\begin{equation}
\label{r.7}
Y^n_t=\xi+\int_t^Tf(r,Y^n_r,Z^n_r)\,dr+\int^T_t\,dV_r-\int_t^T Z^n_r\,dB_r+\int^T_t\,dK^n_r-\int^T_t\,dA^n_r,
\end{equation}
where
\begin{align}
\label{r.8} K^n_t&=n\int_0^t (\bar{Y}^n_r-L_r)^-\,dr+\sum_{0\leq \sigma_{n,i}<t}(Y^n_{\sigma_{n,i}+}+\Delta^+V_{\sigma_{n,i}}-L_{\sigma_{n,i}})^-\nonumber \\
&=:K^{n,*}_t+K^{n,d}_t,\quad t\in[0,T]
\end{align}
and
\begin{align}
\label{r.9} A^n_t&=n\int_0^t (Y^n_r-U_r)^+\,dr+\sum_{0\leq \tau_{n,i}<t}(Y^n_{\tau_{n,i}+}+\Delta^+V_{\tau_{n,i}}-U_{\tau_{n,i}})^+\nonumber \\
&=:A^{n,*}_t+A^{n,d}_t,\quad t\in[0,T]
\end{align}

{\bf Step 1.} We show the convergence of $\{Y^n\}$. By  Step 1 of Theorem \ref{r.6} and Theorem \ref{czw.3},
we know that  $\bar{Y}^n$ is the first component of a solution of RBSDE($\xi$,$f+dV, L^n, U$) with $L^n=L\wedge Y^n$, and $\underline{Y}^n$ is the first component of a solution of  RBSDE($\xi, f+dV, L, U^n$) with $U^n=U\vee Y^n$. As in the Step 1 of the proof of the Theorem \ref{r.6} one can show that the triple $(Y^n,Z^n,K^n-A^n)$ is a solution of  RBSDE$(\xi,f+dV,L^n,U^n)$.
By Proposition \ref{proposition2.2} and Lemma \ref{uw2.0},
\begin{equation}\label{czw.2}
\bar{Y}^n\le Y^n\le\underline{Y}^n,\quad n\ge 1.
\end{equation}
By Theorem \ref{r.6} and Theorem \ref{czw.3},
\begin{equation}\label{czw.4}
Y^n_t\rightarrow Y_t,\quad t\in[0,T].
\end{equation}

{\bf Step 2.} We  show the convergence of $\{Z^n\}$ in measure $dt\otimes P$. For this end, we willl apply Lemma \ref{r.5} to (\ref{r.7}). Since we know that $Y^n_t\rightarrow Y_t$, $t\in[0,T]$, assumption (e) of Lemma \ref{r.5} is satisfied. We are going to check assumptions the remaining assumptions (a)-(d). Let $D^n=V+K^n-A^n$, $n\in\N$. First we will show  that there exists a chain $\{\tau_k\}$ such that
$\sup_{n\ge 0}E((D^n)^+_{\tau_k})^2<\infty$.
If $p>1$, then by (H6), there exists $X\in\mathcal{M}_{loc}+\mathcal{V}^p$, $X\in\mathcal{S}^p$ such that $U\ge X\ge L$ and $\int^T_0 f^-(s,X_s,0)\,ds\in \mathbb L^p$. If $p=1$, then by (H6*), there exists $X$ of class (D), $X\in\mathcal{M}_{loc}+\mathcal{V}^1$, $X\ge L$ and $\int^T_0 f^-(s,X_s,0)\,ds\in \mathbb L^1$. Since the Brownian filtration has the representation property, there exist processes $H\in\mathcal{M}_{loc}$ and $C\in\mathcal{V}^p$ such that
\[
X_t=X_T-\int^T_t\,dC_s-\int^T_t H_s\,dB_s,\quad t\in[0,T].
\]
Set  $X'_t=X_t$, $t\in[0,T)$, $X'_T=\xi$. Then for some $A',K'\in\mathcal{V}^{+,p}$ we have
\begin{equation*}
X'_t=\xi+\int^T_t f(r,X'_r,H_r)\,dr+\int^T_t\,dV_r+\int^T_t\,dK'_r-\int^T_t\,dA'_r-\int^T_t H_r\,dB_r,\quad t\in[0,T].
\end{equation*}
Let $(\hat{X}^n,\hat{H}^n,\hat{K}^n)$ be a solution of the $\mathrm{\underline{R}BSDE}$
\begin{align}
\label{s.1}
\hat{X}^n_t&=\xi+\int^T_t f(r,\hat{X}^n_r,\hat{H}^n_r)\,dr+\int^T_t\,dV_r+\int^T_t\,d\hat{K}^n_r-\int^T_t\,dA'_r-\int^T_t \hat{H}^n_r\,dB_r\nonumber\\
&\quad+\sum_{t\le\tau_{n,i}<T}(\hat{X}^n_{\tau_{n,i}+}+\Delta^+V_{\tau_{n,i}}-U_{\tau_{n,i}})^+,\quad t\in[0,T],
\end{align}
with lower barrier $L^n$, such that if $p>1$, then $\hat{X}^n\in\mathcal{S}^p$, $\hat{H}^n\in\mathcal{H}^p$, $\hat{K}^n\in\mathcal{V}^{+,p}$ and if $p=1$, then $\hat{X}^n$ is of class \textnormal{(D)}, $\hat{H}^n\in\mathcal{H}^q$, $q\in(0,1)$, $\hat{A}^n\in\mathcal{V}^{+,1}$. The existence of the solution follows from  \cite[Theorem 3.18]{KRzS}. Note that $(X',H,K')$ is a solution of $\rm{\underline{R}}$BSDE$(\xi,f+dV-dA',X')$.
Since $L^n\le X'$, by Proposition \ref{proposition2.2} and Lemma \ref{uw2.0}, $\hat{X}^n\le X'\le U$. Thanks to this, we may rewrite (\ref{s.1}) in the  form
\begin{align*}
\hat{X}^n_t&=\xi+\int^T_t f(r,\hat{X}^n_r,\hat{H}^n_r)\,dr+\int^T_t\,dV_r+\int^T_t\,d\hat{K}^n_r-\int^T_t\,dA'_r-n\int^T_t(\hat{X}^n_r-U_r)^+\,dr\\
&\quad-\sum_{t\le\tau_{n,i}<T}(\hat{X}^n_{\tau_{n,i}+}+\Delta^+V_{\tau_{n,i}}-U_{\tau_{n,i}})^+-\int^T_t \hat{H}^n_r\,dB_r,\quad t\in[0,T].
\end{align*}
By Proposition \ref{proposition2.2} and Lemma \ref{uw2.0} again,
\begin{equation}\label{s.2}
\hat{X}^n\le \underline{Y}^n,
\end{equation}
and by Proposition \ref{stw.4.1},
\begin{equation}\label{s.3}
d\hat{K}^n\ge dK^n,\quad n\ge 1.
\end{equation}
Moreover,
\begin{align*}
(\hat{X}^n_{\tau_{n,i}+}+\Delta^+V_{\tau_{n,i}}-U_{\tau_{n,i}})^+
&\le (X'_{\tau_{n,i}+}+\Delta^+V_{\tau_{n,i}}-U_{\tau_{n,i}})^+\nonumber\\
&=(\Delta^+X'_{\tau_{n,i}}+\Delta^+V_{\tau_{n,i}}
+X'_{\tau_{n,i}}-U_{\tau_{n,i}})^+\nonumber\\
&\le(\Delta^+X_{\tau_{n,i}}+\Delta^+V_{\tau_{n,i}})^+\nonumber\\
&\le\Delta^+|C|_{\tau_{n,i}}+\Delta^+|V|_{\tau_{n,i}}.
\end{align*}
Let $(\tilde{X}^n,\tilde{H}^n,\tilde{K}^n)$ be a solution of the $\mathrm{\underline{R}BSDE}$
\begin{align*}
\tilde{X}^n_t&=\xi+\int^T_t f(r,\tilde{X}^n_r,\tilde{H}^n_r)\,dr+\int^T_t\,dV_r+\int^T_t\,d\tilde{K}^n_r-\int^T_t\,dA'_r-n\int^T_t(\tilde{X}^n_r-U_r)^+\,dr\\
&\quad-\int^T_t\,d|C|_r-\int^T_t\,d|V|_r-\int^T_t \tilde{H}^n_r\,dB_r,\quad t\in[0,T],
\end{align*}
with lower barrier $L^n$, such that if $p>1$, then $\tilde{X}^n\in\mathcal{S}^p$, $\tilde{H}^n\in\mathcal{H}^p$, $\tilde{K}^n\in\mathcal{V}^{+,p}$ and if $p=1$, then $\tilde{X}^n$ is of class \textnormal{(D)}, $\tilde{H}^n\in\mathcal{H}^q$, $q\in(0,1)$, $\tilde{K}^n\in\mathcal{V}^{+,1}$. The existence of the solution follows \cite[Theorem 3.18]{KRzS}. By Proposition \ref{proposition2.2} and Lemma \ref{uw2.0}, $\tilde{X}^n\le\hat{X}^n$. By this and  (\ref{s.2}),
\begin{equation}\label{s.5}
\tilde{X}^n\le \underline{Y}^n.
\end{equation}
Moreover, by Proposition \ref{stw.4.1}, $d\hat{K}^n\le d\tilde{K}^n$, which by (\ref{s.3}) implies that
\begin{equation}\label{s.6}
dK^n\le d\tilde{K}^n.
\end{equation}
By \cite[Lemma 4.8]{kl}, there exists a chain $\{\tau'_k\}$ such that
\begin{equation}\label{s.8}
E\big(\sup_{t\le\tau'_k}(|\underline{Y}^1_t|^2+|\overline{Y}^1_t|^2)\big)<\infty,\quad k\ge 1.
\end{equation}
Define
$
\tau''_k=\inf\{t\ge 0;\,\int^t_0|f(r,0,0)|\,dr+\int^t_0\,d|V|_r+\int^t_0 f^-(r,\underline{Y}^1_r,0)\,dr\ge k\}\wedge T
$, $k\in\N$.
Since $\underline{Y}^{n+1}\le\underline{Y}^n$ and $\overline{Y}^n\le\overline{Y}^{n+1}$  by (\ref{czw.2}) we have that
\begin{equation}\label{s.7}
\overline{Y}^1\le Y^n\le\underline{Y}^1.
\end{equation}
Hence
\begin{equation}\label{s.9}
|Y^n|\le|\underline{Y}^1|+|\overline{Y}^1|.
\end{equation}
Set $\tau_k=\tau'_k\wedge\tau''_k$. By (\ref{s.9}) and \cite[Proposition 4.3]{kl},
\begin{align}\label{s.10}
\nonumber
E((\tilde{K}^n_{\tau_k})^2&\le CE\Big(\sup_{t\le\tau_k}\big(|\underline{Y}^1_{\tau_k}|^2\big)+|\overline{Y}^1_{\tau_k}|^2+\Big(\int_0^{\tau_k}|f(r,0,0)|\,dr\Big)^2\\&\quad+\Big(\int_0^{\tau_k}\,d|V|_r\Big)^2
+\Big(\int_0^{\tau_k}|f^-(r,\underline{Y}^1_r,0)|\,dr\Big)^2\Big)<\infty.
\end{align}
That (a) is satisfied now follows from (\ref{s.6}), (\ref{s.10}) and \cite[Lemma 4.8]{kl}. Since $(Y^n,Z^n,K^n-A^n)$ is a solution to RBSDE$(\xi,f+dV,L^n,U^n)$, for $\sigma,\tau\in\Gamma$ such that  $\sigma\le\tau$ we have
\[
\int^{\tau}_{\sigma}(Y_r-Y^n_r)\,dD^{n,*}_r+\sum_{\sigma\le t<\tau}(Y_t-Y^n_t)\Delta^+D^n_t\ge\int^{\tau}_{\sigma}(Y_r-Y^n_r)\,dV^{*}_r+\sum_{\sigma\le t<\tau}(Y_t-Y^n_t)\Delta^+V_t.
\]
From this, (\ref{s.9}) and the Lebesgue dominated convergence theorem we get (b). Assumption  (c) follows from the inequality
$|\Delta^-(Y_t-Y^n_t)|=|\Delta^-Y_t|\le\Delta^-|R|_t+\Delta^-|V|_t$.
Assumption (d) follows from (\ref{s.7}) and Proposition \ref{prop2.5}. Since (a)--(e) are satisfied,   Lemma \ref{r.5} yields
\begin{equation}
\label{eq.wewo3}
Z^n\rightarrow Z,\quad dt\otimes P\mbox{-a.e.}
\end{equation}

{\bf Step 3.} We  will show (\ref{eq.wewo1}) and  (\ref{eq.wewo2}). By (\ref{s.6}), (\ref{s.9}) and \cite[Lemma 4.2]{kl},  if $p>1$, then
\begin{align}\label{w.1}
\nonumber
E\Big(\int^T_0|Z^n_s|^2\,ds\Big)^{{p}/{2}}&\le CE\Big(\sup_{t\le T}(|\underline{Y}^1_t|^p+|\overline{Y}^1_t|^p)+\Big(\int^T_0\,d|V|_s\Big)^p+\Big(\int^T_0|f^-(s,\underline{Y}^1_s,0)|\,ds\Big)^p\\&\quad+\Big(\int^T_0|f(s,0,0)|\,ds\Big)^p
+\Big(\int^T_0\,d\tilde{K}^n_s\Big)^p\Big),
\end{align}
and if $p=1$, then for every $q\in(0,1)$,
\begin{align}\label{w.2}
\nonumber
E\Big(\int^T_0&|Z^n_s|^2\,ds\Big)^{{q}/{2}}\le CE\Big(\sup_{t\le T}(|\underline{Y}^1_t|^q+|\overline{Y}^1_t|^q)+\Big(\int^T_0|f(s,0,0)|\,ds\Big)^q\\&
\quad+\Big(\int^T_0|f^-(s,\underline{Y}^1_s,0)|\,ds\Big)^q+\Big(\int^T_0\,d|V|_s\Big)^q
+\Big(\int^T_0\,d\tilde{K}^n_s\Big)^q\Big).
\end{align}
Let  $(\tilde{X},\tilde{H})$ be a solution of the  BSDE
\begin{align*}
\tilde{X}_t&=\xi+\int^T_t f(r,\tilde{X}_r,\tilde{H}_r)\,dr+\int^T_t\,dV_r-\int^T_t\,dA'_r\\&
\quad-\int^T_t\,d|C|_r-\int^T_t\,d|V|_r-\int^T_t \tilde{H}_r\,dB_r,\quad t\in[0,T],
\end{align*}
such that if $p>1$, then $\tilde{X}\in\mathcal{S}^p$, $\tilde{H}\in\mathcal{H}^p$, $\tilde{K}\in\mathcal{V}^{+,p}$ and if $p=1$, then $\tilde{X}$ is of class \textnormal{(D)}, $\tilde{H}\in\mathcal{H}^q$, $q\in(0,1)$, $\tilde{K}\in\mathcal{V}^{+,1}$. The existence of the solution follows from \cite[Theorem 3.18]{KRzS}. By Proposition \ref{proposition2.2} and Lemma \ref{uw2.0}, $\tilde{X}\le\tilde{X}^n$, so by  (\ref{s.5}),
\begin{equation}\label{w.3}
\tilde{X}\le\tilde{X}^n\le\underline{Y}^1.
\end{equation}
By (\ref{w.3}) and  \cite[Lemma 4.2]{kl}, if $p>1$, then
\begin{align}\label{czwartek.1}
\nonumber
E(\tilde{K}^n_T)^p&\le CE\Big(\sup_{t\le T}(|\underline{Y}^1_t|^p+|\tilde{X}_t|^p)+\Big(\int^T_0\,d|V|_s\Big)^p+\Big(\int^T_0|f^-(s,\underline{Y}^1_s,0)|\,ds\Big)^p\\&\quad+\Big(\int^T_0|f(s,0,0)|\,ds\Big)^p
+\Big(\int^T_0\,dA'_s\Big)^p+\Big(\int^T_0\,d|C|_s\Big)^p\Big),
\end{align}
and if $p=1$, then for every $q\in (0,1)$,
\begin{align}\label{czwartek.2}
\nonumber
E(\tilde{K}^n_T)^q&\le CE\Big(\sup_{t\le T}(|\underline{Y}^1_t|^q+|\tilde{X}_t|^q)+\Big(\int^T_0\,d|V|_s\Big)^q
+\Big(\int^T_0|f^-(s,\underline{Y}^1_s,0)|\,ds\Big)^q\\&\quad+\Big(\int^T_0|f(s,0,0)|\,ds\Big)^q
+\Big(\int^T_0\,dA'_s\Big)^q+\Big(\int^T_0\,d|C|_s\Big)^q\Big).
\end{align}
In case $p>1$, combining (\ref{w.1}) with (\ref{czwartek.1}) we get
\begin{align}\label{czwartek.3}\nonumber
E\Big(\int^T_0|Z^n_r|^2\,dr\Big)^{{p}/{2}}&\le CE\Big(\sup_{t\le T}(|\underline{Y}^1_t|^p+|\overline{Y}^1_t|^p+|\tilde{X}_t|^p)+\Big(\int^T_0\,d|V|_r\Big)^p\\&\quad+\Big(\int^T_0|f^-(r,\underline{Y}^1_r,0)|\,dr\Big)^p\nonumber
+\Big(\int^T_0|f(r,0,0)|\,dr\Big)^p\\&\quad
+\Big(\int^T_0\,dA'_r\Big)^p+\Big(\int^T_0\,d|C|_r\Big)^p\Big)\Big),
\end{align}
In case $p=1$, combining (\ref{w.2}) with (\ref{czwartek.2}) we get
\begin{align}\label{czwartek.4}
\nonumber
E\Big(\int^T_0|Z^n_r|^2\,dr\Big)^{{q}/{2}}&\le CE\Big(\sup_{t\le T}(|\underline{Y}^1_t|^q+|\overline{Y}^1_t|^q|+\tilde{X}_t|^q)+\Big(\int^T_0|f(r,0,0)|\,dr\Big)^q\\
\nonumber&\quad+\Big(\int^T_0|f^-(r,\underline{Y}^1_r,0)|\,dr\Big)^q+\Big(\int^T_0\,d|V|_r\Big)^q\\
&\quad+\Big(\int^T_0\,dA'_r\Big)^q+\Big(\int^T_0\,d|C|_r\Big)^q\Big)\Big)
\end{align}
for $q\in(0,1)$. From (\ref{eq.wewo3}) and (\ref{czwartek.3}), (\ref{czwartek.4})  we easily get  (\ref{eq.wewo1}) and (\ref{eq.wewo2}).

{\bf Step 4.} We will show that if  $\Delta^-R=0$, then  (\ref{eq.wewo1}) and (\ref{eq.wewo2}) hold with $\gamma=2$ and $|Y^n-Y|_p\rightarrow0$.
To this end, we first note that  by (\ref{czw.2}),
\[
 \sup_{t\leq T}||Y^n_t-Y_t|\leq \sup_{t\leq T}\max(|\bar Y^n_t-Y_t|,|\underline{Y}^n_t-Y_t|) .
 \]
By this, Theorem \ref{r.6}  and Theorem \ref{czw.3},
$|Y^n-Y|_p\rightarrow 0$.
Now set $R^n=K^n-A^n$, $n\in\N$, and observe that by \cite[Corollary 5.5]{KRzS}, hypotheses (H1) and  (H2) and the assumption that   $\Delta^-R=0$,
\begin{align*}
\nonumber
\int^T_0|Z_r-Z^n_r|^2\,dr&\le 2\lambda\int^T_0|Y_r-Y^n_r||Z_r-Z^n_r|\,dr+2\int^T_0(Y_r-Y^n_r)\,d(R-R^n)^{c}_r\\
&\quad+2\sum_{0\le t<T}(Y_t-Y^n_t)\Delta^+(R_t-R^n_t)\nonumber\\
&\quad+\sup_{0\le t\le T}\Big|\int^T_t(Y_r-Y^n_r)(Z_r-Z^n_r)\,dB_r\Big|\\
&\le 2\lambda\int^T_0|Y_r-Y^n_r||Z_r-Z^n_r|\,dr+2\int^T_0(Y_r-Y^n_r)\,dR^{c}_r
\\
&\quad+2\sum_{0\le t<T}(Y_t-Y^n_t)\Delta^+R_t+\sup_{0\le t\le T}\Big|\int^T_t(Y_r-Y^n_r)(Z_r-Z^n_r)\,dB_r\Big|.
\end{align*}
Applying the  the Burkholder--Davis--Gundy inequality yields
\[
E\Big(\int^T_0|Z_r-Z^n_r|^2\,dr\Big)^{{p}/{2}}\le C\Big(E\sup_{0\le t\le T} |Y_t-Y^n_t|^p+(E\sup_{0\le t\le T}|Y_t-Y^n_t|^p)^{{1}/{2}}(E|R|^p_T)^{{1}/{2}}\Big).
\]
It is clear that the above inequality implies
(\ref{eq.wewo1}) and (\ref{eq.wewo2})   with $\gamma=2$, which completes the proof.
\end{proof}

\end{document}